\documentclass[reqno, 12pt]{amsart}
\pdfoutput=1
\usepackage[algoruled,lined,noresetcount,norelsize]{algorithm2e}

\usepackage{graphicx}
\usepackage{amsmath,amssymb,amsthm}
\usepackage{mathrsfs}
\usepackage{mathabx}\changenotsign
\usepackage{dsfont}
\usepackage{xcolor}

\usepackage[T1]{fontenc}
\usepackage{lmodern}
\usepackage[babel]{microtype}
\usepackage[english]{babel}

\usepackage{geometry}
\numberwithin{equation}{section}
\numberwithin{figure}{section}

\usepackage{enumitem}
\usepackage{enumerate}
\usepackage{caption}
\usepackage{subcaption}

\theoremstyle{plain}
\newtheorem{thm}{Theorem}[section]

\newtheorem{cor}[thm]{Corollary}
\newtheorem{lemma}[thm]{Lemma}
\newtheorem{rem}[thm]{Remark}

\newtheorem{example}[thm]{Example}

\theoremstyle{definition}
\newtheorem{dfn}[thm]{Definition}

\newcommand{\E}{\mathcal{E}}

\newcommand{\topo}[0]{\mathrm{top}}
\newcommand{\Topo}[0]{\mathrm{TOP}}

\usepackage{tikz}
\usetikzlibrary{decorations.pathreplacing}
\usetikzlibrary{decorations.markings}
\usetikzlibrary{bending}

\usepackage[
style=alphabetic,    
isbn=false,                
pagetracker=true, 
maxbibnames=10,            
minbibnames=3,
maxcitenames=3,            
autocite=inline,           
block=space,               
date=short,                
backend=bibtex
]{biblatex}
\renewbibmacro*{in:}{%
	\iffieldundef{journaltitle}
	{}
	{\printtext{\bibstring{in}\intitlepunct}}}
\begin{document}

\title[S]{Combinatorial proof of Selberg's Integral Formula}
\author[Alexander Haupt]{Alexander Haupt}
\address{(AH) Technische Universit\"at Hamburg, Institut f\"ur Mathematik, Am Schwarzenberg-Campus 3, 21073 Hamburg, Germany }
\email{alexander.haupt@tuhh.de}

\pagestyle{plain}

\date{\today}

\maketitle

\begin{abstract}
	In this paper we present a combinatorial proof of Selberg's integral formula. We start by giving a bijective proof of a Theorem about the number of topological orders of a certain related directed graph. Selberg's Integral Formula then follows by induction. This solves a problem posed by R. Stanley in 2008. Our proof is based on Andersons analytic proof of the formula. As part of the proof we show a further generalisation of the generalised Vandermonde determinant. 
\end{abstract}

\section{Introduction}

\subsection{Selbergs Integral Formula.}
Stanley has stated the following result in his collection of open bijective problems.
\begin{thm}[Problem 27 \cite{Stanley2015BIJECTIVEPP}]
	\label{intro}
	Let $n\ge 2$ and $t \ge 0$. Let $f(n,t)$ be the number of sequences of length $n+2t \binom{n}{2}$ with $n$ symbols named $x$ and $2t$ symbols named $a_{ij}$'s where $1 \le i < j \le n$, such that each $a_{ij}$ occurs between the $i$th and the $j$th copy of $x$ in the sequence. Then we have
	$$f(n,t)= \frac{(n+2t \binom{n}{2})!}{n! (2t)!^{\binom{n}{2}}} \prod_{j=1}^n \frac{((j-1)t)!^2 (jt)!}{t!(1+(n+j-2)t)!}.$$
\end{thm}
You can find a slightly more general version, together with a sketch of proof, is in his book \enquote{Enumerative Combinatorics} \cite{MR2868112}[Chapter 1, Problem 11]. The striking fact about this result is that the only known proof is non-combinatorial in that it uses an analytical result known as the Selberg's Integral Formula. He stresses that a combinatorial proof would be very interesting, as the problem definition and solution are both of combinatorial nature. We now give a very brief sketch of the analytical proof:
\begin{proof}[Sketch of proof using Selberg's Integral Formula]
	For each of the $n+2t \binom{n}{2}$ symbols choose a point independently and uniformly at random from $[0,1]$. The order of the points in $[0,1]$ gives a sequence of symbols, where every permutation is equally likely. We can deduce that the number of sequences that follow the rules above equals:
	$$\frac{\left (n+2t \binom{n}{2}\right )!}{n!(2t)!^{\binom{n}{2}}} \int_0^1 \cdots \int_0^1 \prod_{j>i} |x_j-x_i|^{2t} dx_1 \ldots dx_n.
	$$
	Now for the last step we evaluate the integral using the Selberg's Integral Formula, which immediately gives the result.
\end{proof}

Selberg's integral formula in its general form was first proved by Selberg in 1944 and then was rediscovered and generalised a number of times since then. 
\begin{thm}[Selberg's integral formula]
	\label{selberg}
	Define
	$$S_n(a,b,c):=\frac{1}{n!} \int_0^1 \cdots \int_0^1 \prod_i x_i^{a-1} (1-x_i)^{b-1} \prod_{i<j} |x_i-x_j|^{2c} dx_1 \ldots dx_n.$$
	Then if the integral converges absolutely, we have
	$$S_n(a,b,c)=\prod _{j=0}^{n-1} \frac{ \Gamma(a+c j) \Gamma(b+c j)\Gamma(c+cj)}{\Gamma(c) \Gamma(a+b+(j+n-1)c)}.$$
\end{thm}

With the aim of finding a combinatorial proof of Theorem \ref{intro}, Kim et al defined Selberg pages and Young pages in 2014 \cite{kim:hal-01207598} and proposed that one could find a combinatorial proof via these constructs. However, they showed that not for all Young-pages there exist \enquote{nice} product formulas. This means that arguments via induction on Young-pages will most likely fail.

In this paper we like to argue about topological orderings of directed acyclic graphs, which we define now. For a more detailed introduction of topological orderings, see \cite{knuth}. 
\begin{dfn}
	Given a directed acyclic graph $G=(V,E)$ on $n$ vertices we define a \textit{topological ordering} as a bijection $f : V \rightarrow [n]$ such that for each directed edge $i \rightarrow j$ we have $f(i) > f(j)$.
	We denote by $\Topo(G)$ the set of of topological orderings of $G$ and define $\topo(G):=|\Topo(G)|$.
\end{dfn}
We use graphs instead of sequences, as they are easier to visualise. However, both views are essentially equivalent, as rules for these sequences can be translated into directed graphs, and vice versa. For example the rule that a symbol named $a_{ij}$ must be between $x_i$ and $x_j$ can be expressed with two edges $x_j \rightarrow a_{ij}$ and $a_{ij} \rightarrow x_i$, which we sometimes combine to $x_j \rightarrow a_{ij} \rightarrow x_i$. The following special graph captures the rules for our sequences.
\begin{dfn}    
	\label{defSGraph}
	For integers $n,a,b,c \ge 1$ let $G_S(n,a,b,c)=(V_S,E_S)$ be the graph with vertex set:
	\begin{align*}
	V_S = &\ \{x_1,x_2,\ldots,x_n\}\\
	&\cup\ \{p_{i,k} :\ \forall i \in [n]\ \forall k \in [a-1]\}\\
	&\cup\ \{q_{i,k} :\ \forall i \in [n]\ \forall k \in [b-1]\}\\
	&\cup\ \{r_{i,j,k} :\ \forall i<j \in [n]\ \forall k \in [2c]\}
	\end{align*}
	and edge set
	\begin{align*}E_S = &\ \{x_{i+1} \rightarrow x_i : \forall i \in [n-1]\} \\
	&\cup\ \{x_i \rightarrow p_{i,k} :\ \forall i \in [n]\ \forall k \in [a-1]\}\\
	&\cup\ \{q_{i,k} \rightarrow x_i :\ \forall i \in [n]\ \forall k \in [b-1]\}\\
	&\cup\ \{x_j \rightarrow r_{i,j,k} \rightarrow x_i :\ \forall i<j \in [n]\ \forall k\in [2c] \}.\end{align*}  
\end{dfn}
\begin{figure}
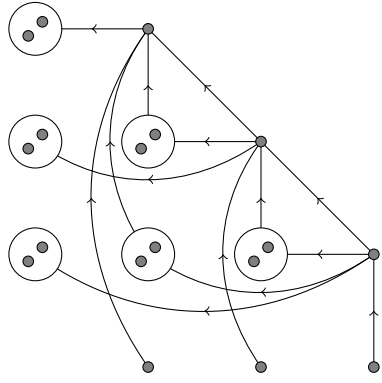

	\centering
	\tikz[baseline=(p)]{
		\tikzset{scale=0.750000,yscale=-1,shorten >=2pt,decoration={markings,mark=at position 0.5 with {\arrow{>}}}}
		\tikzset{every node/.style={circle,draw,solid,fill=black!50,inner sep=0pt,minimum width=4pt}}
		\draw[postaction={decorate}] (2,2) to (0,0);
		\draw[postaction={decorate}] (4,4) to (2,2);
		
		\draw[postaction={decorate}] (2,2) to (0,2);
		\draw[postaction={decorate}] (0,2) to (0,0);
		\draw[postaction={decorate},bend right=35] (4,4) to (0,4);
		\draw[postaction={decorate},bend right=35] (0,4) to (0,0);
		\draw[postaction={decorate}] (4,4) to (2,4);
		\draw[postaction={decorate}] (2,4) to (2,2);
		
		\draw[postaction={decorate},bend right=35] (0,6) to (0,0);
		\draw[postaction={decorate},bend right=35] (2,6) to (2,2);
		\draw[postaction={decorate}] (4,6) to (4,4);
		\draw[postaction={decorate}] (0,0) to (-2,0);
		\draw[postaction={decorate},bend right=35] (2,2) to (-2,2);
		\draw[postaction={decorate},bend right=35] (4,4) to (-2,4);
		
		\draw (-2,0) node[draw=black,fill=white,minimum width=20pt] {};
		\draw (-2,2) node[draw=black,fill=white,minimum width=20pt] {};
		\draw (-2,4) node[draw=black,fill=white,minimum width=20pt] {};
		\draw (0,2) node[draw=black,fill=white,minimum width=20pt] {};
		\draw (0,4) node[draw=black,fill=white,minimum width=20pt] {};
		\draw (2,4) node[draw=black,fill=white,minimum width=20pt] {};
		
		\draw (0,0) node[draw=black,fill=black!50] {};
		\draw (2,2) node[draw=black,fill=black!50] {};
		\draw (4,4) node[draw=black,fill=black!50] {};
		\draw (-0.125,2.125) node[draw=black,fill=black!50] {};
		\draw (0.125,1.875) node[draw=black,fill=black!50] {};
		\draw (-0.125,4.125) node[draw=black,fill=black!50] {};
		\draw (0.125,3.875) node[draw=black,fill=black!50] {};
		\draw (1.875,4.125) node[draw=black,fill=black!50] {};
		\draw (2.125,3.875) node[draw=black,fill=black!50] {};
		\draw (0,6) node[draw=black,fill=black!50] {};
		\draw (2,6) node[draw=black,fill=black!50] {};
		\draw (4,6) node[draw=black,fill=black!50] {};
		\draw (-2.125,0.125) node[draw=black,fill=black!50] {};
		\draw (-1.875,-0.125) node[draw=black,fill=black!50] {};
		\draw (-2.125,2.125) node[draw=black,fill=black!50] {};
		\draw (-1.875,1.875) node[draw=black,fill=black!50] {};
		\draw (-2.125,4.125) node[draw=black,fill=black!50] {};
		\draw (-1.875,3.875) node[draw=black,fill=black!50] {};
		\coordinate (p) at ([yshift=1ex]current bounding box.center);
	}
	\captionof{figure}{The directed graph $G_S(3,3,2,1)$}
	\label{fig:GS}
\end{figure}

See Figure \ref{fig:GS} for a drawing of such a graph, where we have grouped some vertices and edges together. Theorem \ref{selberg} for integers $n,a,b,c \ge 1$ can be stated combinatorially in terms of topological orderings as follows:
\begin{thm}[Combinatorial interpretation]
	\label{combSelberg}
	For integers $n, a,b,c \ge 1$, we have:
	\begin{align*}&\topo(G_S(n,a,b,c))\\
	&=\left (n(a+b-1) + \binom{n}{2}\cdot 2c\right )!\prod _{j=0}^{n-1} \frac{(a+c j-1)! (b+c j-1)!(c+cj-1)! }{(c-1)! (a+b+(j+n-1)c-1)!}.\end{align*}
\end{thm}
Note that for $a=b=1$ this implies Theorem \ref{intro}, and in general implies the more general version as seen in \cite{MR2868112}[Chapter 1, Problem 11].

The proof of Theorem \ref{combSelberg} is essentially a combinatorial analogue of Anderson's proof \cite{Anderson1991} of Theorem \ref{selberg}. Most of the steps in his proof have combinatorial interpretations, and so the problem basically reduces to finding a combinatorial proof of Theorem \ref{crux}.

\subsection{Structure of the paper.}
In section \ref{chapterCrux} we investigate the number of topological orderings of a special graph related to the graph $G_S$. We state and prove Theorem \ref{crux}, first showing the motivation, then turning the ideas into a bijection. In section \ref{sectionSelberg} we use Theorem \ref{crux} to prove Theorem \ref{combSelberg} by induction. For better understanding we also give an example, which shows how Theorem \ref{crux} is applied. We finish the paper with remarks and open questions.

\section{Bijective proof of Theorem \ref{crux}}
\label{chapterCrux}

\subsection{Motivation}
In this section we will investigate the topological orders of the following graph.
\begin{dfn}
	\label{defGraphX}
	For fixed $\alpha=(\alpha_1,\ldots,\alpha_{n})$ with $\alpha_i \ge 0$, 
	let $G_X(\alpha)=(V_A,E_A)$ be the directed graph with vertex set:
	\begin{align*}
	V_A = \{u_1,u_2,\ldots,u_n\}&\\
	\cup\ \{w_{i,j,k} :\ &\forall i<j \in [n], \forall k\in [\alpha_i\cdot \alpha_j] \}
	\end{align*}
	and edge set
	\begin{align*}E_A = \{u_{i+1} \rightarrow u_i : \forall i \in [n-1]\}& \\
	\cup\ \{u_j \rightarrow w_{i,j,k} \rightarrow u_i :\ &\forall i<j \in [n], \forall k\in [\alpha_i \cdot \alpha_j] \}.\end{align*}
\end{dfn}

Figure \ref{fig:GX} shows an example of such a graph for $\alpha=(2,1,3)$.
\begin{figure}
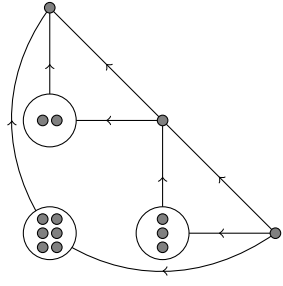

	\centering
	\tikz[baseline=(p)]{
		\tikzset{scale=0.750000,yscale=-1,shorten >=2pt,decoration={markings,mark=at position 0.5 with {\arrow{>}}}}
		\tikzset{every node/.style={circle,draw,solid,fill=black!50,inner sep=0pt,minimum width=4pt}}
		\draw[postaction={decorate}] (2,2) to (0,0);
		\draw[postaction={decorate}] (4,4) to (2,2);
		
		\draw[postaction={decorate}] (2,2) to (0,2);
		\draw[postaction={decorate}] (0,2) to (0,0);
		\draw[postaction={decorate},bend right=35] (4,4) to (0,4);
		\draw[postaction={decorate},bend right=35] (0,4) to (0,0);
		\draw[postaction={decorate}] (4,4) to (2,4);
		\draw[postaction={decorate}] (2,4) to (2,2);

		\draw (0,2) node[draw=black,fill=white,minimum width=20pt] {};
		\draw (0,4) node[draw=black,fill=white,minimum width=20pt] {};
		\draw (2,4) node[draw=black,fill=white,minimum width=20pt] {};
		
		\draw (0,0) node[draw=black,fill=black!50] {};
		\draw (2,2) node[draw=black,fill=black!50] {};
		\draw (4,4) node[draw=black,fill=black!50] {};
		\draw (-0.125,2) node[draw=black,fill=black!50] {};
		\draw (0.125,2) node[draw=black,fill=black!50] {};
		\draw (-0.125,3.75) node[draw=black,fill=black!50] {};
		\draw (0.125,3.75) node[draw=black,fill=black!50] {};
		\draw (-0.125,4) node[draw=black,fill=black!50] {};
		\draw (0.125,4) node[draw=black,fill=black!50] {};
		\draw (-0.125,4.25) node[draw=black,fill=black!50] {};
		\draw (0.125,4.25) node[draw=black,fill=black!50] {};
		\draw (2,3.75) node[draw=black,fill=black!50] {};
		\draw (2,4) node[draw=black,fill=black!50] {};
		\draw (2,4.25) node[draw=black,fill=black!50] {};
		\coordinate (p) at ([yshift=1ex]current bounding box.center);
	}
	\captionof{figure}{The directed graph $G_X(2,1,3)$}
	\label{fig:GX}
\end{figure}

\begin{thm}
	\label{crux}
	Fix $\alpha_1,\ldots,\alpha_{2n-1} \ge 0$ with $\alpha_i=1$ for $i$ even and let $G_A = G_X(\alpha_1,\alpha_2,\ldots,\alpha_{2n-1})$ and $G_B = G_X(\alpha_1+1,\alpha_3+1,\ldots,\alpha_{2n-1}+1)$. 
	Note that both graphs have $$N:=n+\sum _{1\le i < j \le n} (\alpha_{2 i-1}+1) (\alpha_{2 j-1}+1)$$ vertices.
	Fix any distinct $p_1,p_3,\ldots,p_{2n-1} \in [N]$ and let 
	
	$$A:=\{f \in \Topo(G_A) : f(u_i)=p_i\ \forall i\ \text{odd}\}$$
	$$B:=\{f \in \Topo(G_B) : f(u_i)=p_{2i-1}\ \forall i\}.$$
	In other words, in $B$ we have fixed labels of the diagonal. In $A$ we have fixed every other label of the diagonal. We have a bijection
	$$[(\sum \alpha_i)!] \times A \leftrightarrow \prod [\alpha_i!] \times B.$$
\end{thm}

This Theorem is essentially a combinatorial interpretation of the following result.
\begin{thm}
	\label{cruxIntegral}
	Fix $\alpha_1,\ldots,\alpha_{2n-1} \ge 0$ with $\alpha_i=1$ for $i$ even. We have
	\begin{align*}
	\left (\sum_{i=1}^{2n-1} \alpha_i \right)! \int_{x_1}^{x_3} \ldots \int_{x_{2n-3}}^{x_{2n-1}} \prod_{i<j} (x_j-x_i)^{\alpha_i \alpha_j} dx_{2n-2}\ldots dx_2
	=\prod_{i=1}^{2n-1} (\alpha_i!) \prod_{\substack{i<j\\\text{both odd}}} (x_j-x_i)^{(\alpha_i+1) (\alpha_j+1)}.
	\end{align*}
\end{thm}
Here we can divide through the common factor $\prod_{\substack{i<j\\\text{both odd}}} (x_j-x_i)^{\alpha_i \alpha_j}$, which results in Corollary \ref{cruxIntegralCorollary}, which is used in Andersons Proof of Theorem \ref{selberg} and proved there using a cunning change of variables. In this section, we show a new idea, which does not rely on change of variables and which can then be turned into a bijection..
\begin{cor}
	\label{cruxIntegralCorollary}
	Fix $\alpha_1,\ldots,\alpha_{2n-1} \ge 0$ with $\alpha_i=1$ for $i$ even. We have
	\begin{align*}
	\int_{x_1}^{x_3} \ldots \int_{x_{2n-3}}^{x_{2n-1}} \prod_{\substack{i<j\\\text{not both odd}}} (x_j-x_i)^{\alpha_i \alpha_j} dx_{2n-2}\ldots dx_2
	=\frac{\prod \alpha_i!}{(\sum \alpha_i)!} \prod_{\substack{i<j\\\text{both odd}}} (x_j-x_i)^{\alpha_i+\alpha_j+1}.
	\end{align*}
\end{cor}
\begin{example}
	\label{explCrux}
	For $\alpha=(2,1,3)$ we have
	$$\int_{x_1}^{x_3} (x_2-x_1)^2 (x_3-x_2)^3 (x_3-x_1)^6 \, dx_2 = \frac{1}{60} (x_1-x_3)^{12}$$
	or after dividing through the common factor:
	$$\int_{x_1}^{x_3} (x_2-x_1)^2 (x_3-x_2)^3 \, dx_2 = \frac{1}{60} (x_1-x_3)^6$$
\end{example}

We want to prove Example \ref{explCrux} without a bijection first, to show the idea of the proof. Most of the effort in this section is spent writing the proof as a bijection.
\begin{proof}[Proof of Example \ref{explCrux}]
	We write the expression as a determinant of a matrix, in which $x_2$ appears in one row only. This allows us to integrate each entry in that row separately.
	\begingroup
	\allowdisplaybreaks
	\begin{align*}
	&\int_{x_1}^{x_3} (x_2-x_1)^2 (x_3-x_2)^3 \, dx_2\\
	&=\int_{x_1}^{x_3} (x_3-x_1)^{-6}\cdot \det \begin{pmatrix}
	1&x_1&x_1^2&x_1^3&x_1^4&x_1^5\\
	0&1&2x_1&3x_1^2&4x_1^3&5x_1^4\\
	1&x_2&x_2^2&x_2^3&x_2^4&x_2^5\\
	1&x_3&x_3^2&x_3^3&x_3^4&x_3^5\\
	0&1&2x_3&3x_3^2&4x_3^3&5x_3^4\\
	0&0&1&3x_3&6x_3^2&10x_3^3
	\end{pmatrix}\, dx_2\\
	&=(x_3-x_1)^{-6} \cdot\det \begin{pmatrix}
	1&x_1&x_1^2&x_1^3&x_1^4&x_1^5\\
	0&1&2x_1&3x_1^2&4x_1^3&5x_1^4\\
	x_3-x_1&\frac{1}{2}(x_3^2-x_1^2)&\frac{1}{3}(x_3^3-x_1^3)&\frac{1}{4}(x_3^4-x_1^4)&\frac{1}{5}(x_3^5-x_1^5)&\frac{1}{6}(x_3^6-x_1^6)\\
	1&x_3&x_3^2&x_3^3&x_3^4&x_3^5\\
	0&1&2x_3&3x_3^2&4x_3^3&5x_3^4\\
	0&0&1&3x_3&6x_3^2&10x_3^3
	\end{pmatrix}\\
	&=(x_3-x_1)^{-6}\cdot \frac{2!3!}{6!} \cdot \det \begin{pmatrix}
	1&x_1&x_1^2&x_1^3&x_1^4&x_1^5&x_1^6\\
	0&1&2x_1&3x_1^2&4x_1^3&5x_1^4&6x_1^5\\
	0&0&1&3x_1&6x_1^2&10x_1^3&15x_1^4\\
	1&x_3&x_3^2&x_3^3&x_3^4&x_3^5&x_3^6\\
	0&1&2x_3&3x_3^2&4x_3^3&5x_3^4&6x_3^5\\
	0&0&1&3x_3&6x_3^2&10x_3^3&15x_3^4\\
	0&0&0&1&4x_3&10x_3^2&20x_3^3
	\end{pmatrix}\\
	&=\frac{1}{60} \cdot (x_3-x_1)^{6}
	\end{align*}
	\endgroup
\end{proof}
The first and last matrix in above proof can be evaluated using the Generalised Vandermonde Determinant, as found for example in \cite{10.2307/2690290}. We state it next. 
\begin{thm}[Generalised Vandermonde Determinant]
	\label{generalVandermonde}
	Fix $\alpha=(\alpha_1,\ldots,\alpha_n)$ with $\alpha_i \ge1$. Let $m = \sum \alpha_i$ and $M(\alpha,m,x)$ be the following $\alpha \times m$ matrix:
	$$M(\alpha,m,x) := \begin{pmatrix}
	1&      x&      x^2&  \ldots& \binom{\alpha-1}{0}x^{\alpha-1}&\ldots&\binom{m-1}{0}x^{m-1}\\
	0&      1&      2x&   \ldots& \binom{\alpha-1}{1}x^{\alpha-2}&\ldots&\binom{m-1}{1}x^{m-2}\\
	0&      0&      1&      \ldots& \binom{\alpha-1}{2}x^{\alpha-3}&\ldots&\binom{m-1}{2}x^{m-3}\\
	\ldots& \ldots& \ldots& \ldots& \ldots&\ldots&\\
	0&      0&      0&      \ldots&      1&\ldots&\binom{m-1}{\alpha-1}x^{m-\alpha}\\
	\end{pmatrix}.$$
	Then we have:
	$$\det \begin{pmatrix}
	M(\alpha_1,m,x_1)\\
	\ldots\\
	M(\alpha_n,m,x_n)
	\end{pmatrix} = \prod_{i<j} (x_j-x_i)^{\alpha_i \alpha_j}.$$
\end{thm}
Note that for $\alpha_i = 1\ \forall i$ this reduces to the usual Vandermonde determinant. Before we start with the bijective proof, we want to highlight the following result: Gessel found a bijective proof of the Vandermonde determinant \cite{gessel79}, by comparing terms in the expansion. For example 
$$(x_2-x_1)(x_3-x_1)(x_3-x_2)=\det  \begin{pmatrix}
1&x_1&x_1^2\\
1&x_2&x_2^2\\
1&x_3&x_3^2
\end{pmatrix},
$$
where the terms $x_1 x_2 x_3$ and $-x_1 x_2 x_3$ on the left hand side cancel. We will reformulate this using signed sets and sijections and then generalise it to the Generalised Vandermonde Determinant above.

\subsection{Signed sets and sijections.}
In a recent paper by Fischer and Konvalinka \cite{amsBijection} the concept of signed sets and sijections was introduced. Sijections take the role of bijections for signed sets. We give a short introduction here, but do change the notation slightly, as we also want to use weight functions.

A \emph{signed set} is a finite set $S$ together with a weight function $w_S$, where we allow positive and negative weights. In this paper all weights are positive or negative products of formal variables. Define the weight of the whole set $w(S):=\sum_{s\in S} w(s)$.

A \emph{sijection} $f$ between signed sets $S$ and $T$ is an involution on the set $S \sqcup T$, where $\sqcup$ denotes the disjoint union, such that for~$x \in S \sqcup T$:
$$w(f(x))= \begin{cases}-w(x),&\text{ if }(x \in S\text{ and }f(x)\in S)\text{ or }(x \in T\text{ and }f(x)\in T)\\
w(x),&\text{ otherwise.}\end{cases}$$

The motivation behind this definition is the following: If we have a sijection between $A$ and $B$ then $w(A)=w(B)$.  This is analogous to the fact that if we have a bijection between $A$ and $B$, then $|A|=|B|$. In \cite{amsBijection} they had $w(s)\in \{-1,1\}$, in which case we call the signed set \emph{unweighted}.

If $A$ is a signed set, we define $-A$ as a copy of $A$, except we have $w_{-A}(a):=-w_{A}(a)$.

If $A$ and $B$ are signed sets, define $A+B$ as the set $A \sqcup B$ together with the weight function $$w_{A+B}(x) := \begin{cases}w_A(x),&\text{if }x\in A\\w_B(x),&\text{if }x\in B.\end{cases}$$

Similarly we define $A-B$ as the signed set $A+(-B)$.

\begin{example}
	\label{sijectionExample}
	Let $A$ and $B$ be the following weighted signed sets. Here the symbols are arbitrary, but their shape is chosen in accordance to their weight.
	\begin{align*}
	A &= \{\bullet, \blacksquare\} \times \{\circ,\square\}\\
	&= \{(\bullet,\circ),(\blacksquare,\circ),(\bullet,\square),(\blacksquare,\square)\}\\
	B &= (\{\circledcirc\},\{\boxminus\})
	\end{align*}
	with weights:
	\begin{align*}
	w(\bullet)&=x,
	&w(\blacksquare)&=y,
	&w(\circ)&=x,\\
	w(\square)&=-y,
	&w(\circledcirc)&=x^2,
	&w(\boxminus)&=-y^2
	\end{align*}
	and implicit weights:
	\begin{align*}
	w((\bullet,\circ))&=x^2,
	&w((\blacksquare,\circ))&=x y,\\
	w((\bullet,\square))&=-x y,
	&w((\blacksquare,\square))&=-y^2.
	\end{align*}
	
	Then the following involution is a sijection:
	\begin{align*}
	(\bullet,\circ)&\leftrightarrow \circledcirc,
	&(\blacksquare,\square)&\leftrightarrow \boxminus,
	&(\blacksquare,\circ)&\leftrightarrow (\bullet,\square).
	\end{align*}
	
	Later in the paper, we will not use different symbols for every single object, as seen above, but instead write the weights immediately:
	\begin{align*}
	A &= \{x, y\} \times \{x,-y\}\\
	B &= \{x^2,-y^2\}
	\end{align*}
	From this the underlying reason why a sijection exists is clearer also:
	$$w(A)=(x+y) (x-y)=x^2-y^2=w(B).$$
\end{example}

\subsection{Binomial coefficients.}
\begin{dfn}
	Let $x_1,x_2,\ldots,x_n$ be formal variables and $k$ any integer. Define the following weighted signed set
	$$B((x_1,x_2,\ldots,x_n),k):=\{(a_1,\ldots,a_n) \colon a_i \ge 0, \sum a_i = k\}$$
	with weight $w(a)=\prod x_i^{a_i}$. For $k < 0$ the set is empty.
\end{dfn}
\begin{example}
	The weighted signed set $B((x_1, x_1, x_2), 2)$ has $\binom{n+k-1}{k}=\binom{4}{2}=6$ elements, for example by the stars and bars method, and their weights are: $x_1^2,x_1^2,x_1^2,x_1 x_2,x_1 x_2,x_2^2$.
\end{example}

\begin{lemma}
	\label{lemmaBinomial1}
	Fix any integers $1 \le q \le j$.
	We have a weight-preserving bijection
	$$[q] \times B((\underbrace{x,\ldots,x}_{q+1\text{ times}})),j-q) \leftrightarrow[j] \times B((\underbrace{x,\ldots,x}_{q\text{ times}}),j-q).$$
\end{lemma}
\begin{proof}
	All elements have weight $x^{j-q}$, so we can ignore the weights for our bijection.
	Given an element $(\ell,(a_1,\ldots,a_{q+1}))$ with $\ell \in [q]$ and $\sum_{i\in [q+1]} a_i = j-q$, we create a $q$-tuple by merging $a_\ell$ and $a_{\ell+1}$ as follows: $(a_1,\ldots,a_{\ell-1},a_\ell + a_{\ell+1},a_{\ell+2},\ldots,a_{q+1})$.
	Of course the total sum is unchanged and still equals $j-q$. Additionally we record $\ell+\sum_{i\le \ell} a_i \in [j]$, as to be able to undo the merging step. This completes our bijection:
	$$(\ell,(a_1,\ldots,a_{q+1})) \leftrightarrow (\ell+\sum_{i\le \ell} a_i, (a_1,\ldots,a_{\ell-1},a_\ell + a_{\ell+1},a_{\ell+2},\ldots,a_{q+1})).$$
\end{proof}
\begin{example}
	For $j=4, q=2$, Lemma \ref{lemmaBinomial1} gives the following bijection:
	\begin{align*}
	(1, (2,0,0)) \leftrightarrow (3 ,(2,0)),\quad
	(2, (2,0,0)) \leftrightarrow (4 ,(2,0)),\quad
	(1, (1,1,0)) \leftrightarrow (2 ,(2,0))\\
	(2, (1,1,0)) \leftrightarrow (4 ,(1,1)),\quad
	(1, (1,0,1)) \leftrightarrow (2 ,(1,1)),\quad
	(2, (1,0,1)) \leftrightarrow (3 ,(1,1))\\
	(1, (0,2,0)) \leftrightarrow (1 ,(2,0)),\quad
	(2, (0,2,0)) \leftrightarrow (4 ,(0,2)),\quad
	(1, (0,1,1)) \leftrightarrow (1 ,(1,1))\\
	(2, (0,1,1)) \leftrightarrow (3 ,(0,2)),\quad
	(1, (0,0,2)) \leftrightarrow (1 ,(0,2)),\quad
	(2, (0,0,2)) \leftrightarrow (2 ,(0,2))
	\end{align*}
\end{example}

We need one more lemma. By setting $k=2,m=0$, one can see that it is a generalisation of Example \ref{sijectionExample}.
\begin{lemma}
	\label{lemmaBinomial2}
	Fix integers $k\ge1, m\ge0$ and formal variables $x_i,x_j, y_1,\ldots,y_m$. Then we have a weight-preserving sijection
	$$B((x_j,y_1,\ldots,y_m), k)-B((x_i,y_1,\ldots,y_m), k) \leftrightarrow \{x_j,-x_i\} \times B((x_i,x_j,y_1,\ldots,y_m), k-1).$$
\end{lemma}
\begin{proof}
	Let $A:=B((x_i,x_j,y_1,\ldots,y_m), k)$.
	We clearly have
	$$\{a \in A: a_1=0\} + \{a \in A: a_1\ge 1\} \leftrightarrow A \leftrightarrow \{a \in A: a_2=0\} + \{a \in A: a_2\ge 1\}.$$
	Rearranging gives the sijection
	$$\{a \in A: a_1=0\} - \{a \in A: a_2=0\}\leftrightarrow  \{a \in A : a_2 \ge 1\} - \{a \in A : a_1 \ge 1\}.$$
	
	We also have a bijection $\{a \in A: a_1=0\} \leftrightarrow B((x_j,y_1,\ldots,y_m), k)$ by simply discarding the first entry in the tuple, which is always $0$.
	Similarly we have $\{a \in A: a_2=0\} \leftrightarrow B((x_i,y_1,\ldots,y_m), k)$.
	
	On the other hand we have a bijection $\{a \in A : a_2 \ge 1\} \leftrightarrow \{x_j\} \times B((x_i,x_j,y_1,\ldots,y_m), k-1)$ by decreasing $a_2$ by $1$, but instead multiply by a singleton set with an element with weight $x_j$ to keep the bijection weight-preserving. Similarly we have $\{a \in A : a_1 \ge 1\} \leftrightarrow \{x_i\} \times B((x_i,x_j,y_1,\ldots,y_m), k-1)$.
	This completes the proof, as $\{x_j\}-\{x_i\}=\{x_j,-x_i\}$.
\end{proof}

\subsection{Topological orders.}
We start this subsection by showing a sijection between unweighted signed sets of topological orders.
\begin{lemma}
	\label{lemmaTrickOne}
	Let $G$ be any directed graph with three vertices $b,g,v$ with $v \rightarrow g$, $v \rightarrow b$ and $b \rightarrow g$. Let $G_1$ be a copy of $G$ without the edge $b \rightarrow g$ and $G_2$ a copy of $G$ without $v \rightarrow b$ and $b \rightarrow g$ but instead $g \rightarrow b$.
	
	We have a sijection
	$$\Topo(G) \leftrightarrow \Topo(G_1) - \Topo(G_2).$$
\end{lemma}
\begin{proof}
	Any topological order $f$ of $G_1$ either has $f(b) > f(g)$ or $f(b) < f(g)$. In the first case, $f$ is also a topological order of $G$, while in the second case it is instead a topological order of $G_2$. On the other hand, any topological order of $G$ or $G_2$ is also a topological order of $G_1$.
\end{proof}
Comparing the sizes of the signed sets on both sides yields the equation $\topo(G)=\topo(G_1)-\topo(G_2)$. Or graphically:
\begin{equation}
\topo\left( \tikz[baseline=(p)]{
	\tikzset{scale=0.4,yscale=-1,shorten >=1pt,decoration={markings,mark=at position 0.5 with {\arrow{>}}}}
	\tikzset{every node/.style={circle,draw,solid,fill=black!50,inner sep=0pt,minimum width=4pt}}
	\draw (1.5,1) ellipse (3cm and 2cm);
	\draw[postaction={decorate}] (1,1) to (0,0);
	\draw[postaction={decorate}] (1,1) to (-0.5,1.5);
	\draw[postaction={decorate}] (-0.5,1.5) to (0,0);
	\draw (0,0) node[draw=black,fill=green] {};
	\draw (1,1) node[draw=black,fill=black!50] {};
	\draw (-0.5,1.5) node[draw=black,fill=blue] {};
	\coordinate (p) at ([yshift=1ex]current bounding box.center);
}\right) = \topo \left( \tikz[baseline=(p)]{
	\tikzset{scale=0.4,yscale=-1,shorten >=1pt,decoration={markings,mark=at position 0.5 with {\arrow{>}}}}
	\tikzset{every node/.style={circle,draw,solid,fill=black!50,inner sep=0pt,minimum width=4pt}}
	\draw (1.5,1) ellipse (3cm and 2cm);
	\draw[postaction={decorate}] (1,1) to (0,0);
	\draw[postaction={decorate}] (1,1) to (-0.5,1.5);
	\draw (0,0) node[draw=black,fill=green] {};
	\draw (1,1) node[draw=black,fill=black!50] {};
	\draw (-0.5,1.5) node[draw=black,fill=blue] {};
	\coordinate (p) at ([yshift=1ex]current bounding box.center);
}\right)-\topo\left( \tikz[baseline=(p)]{
	\tikzset{scale=0.4,yscale=-1,shorten >=1pt,decoration={markings,mark=at position 0.5 with {\arrow{>}}}}
	\tikzset{every node/.style={circle,draw,solid,fill=black!50,inner sep=0pt,minimum width=4pt}}
	\draw (1.5,1) ellipse (3cm and 2cm);
	\draw[postaction={decorate}] (1,1) to (0,0);
	\draw[postaction={decorate}] (0,0) to (-0.5,1.5);
	\draw (0,0) node[draw=black,fill=green] {};
	\draw (1,1) node[draw=black,fill=black!50] {};
	\draw (-0.5,1.5) node[draw=black,fill=blue] {};
	\coordinate (p) at ([yshift=1ex]current bounding box.center);
}\right).
\end{equation}

\begin{example}
	\label{exampleTrick}
	We now apply this expansion on all three vertices below the diagonal of $G_\mathcal{X}(1,1,1)$ as follows:
	$$\begin{aligned}
	\topo\left(\tikz[baseline=(p)]{
		\tikzset{scale=0.500000,yscale=-1,shorten >=2pt,decoration={markings,mark=at position 0.5 with {\arrow{>}}}}
		\tikzset{every node/.style={circle,draw,solid,fill=black!50,inner sep=0pt,minimum width=4pt}}
		\draw[postaction={decorate}] (2,2) to (1,1);
		\draw[postaction={decorate}] (3,3) to (2,2);
		\draw[postaction={decorate}]  (2,2) to (1,2);
		\draw[postaction={decorate}]  (1,2) to (1,1);
		\draw[postaction={decorate},bend right=30]  (3,3) to (1,3);
		\draw[postaction={decorate},bend right=30]  (1,3) to (1,1);
		\draw[postaction={decorate}]  (3,3) to (2,3);
		\draw[postaction={decorate}]  (2,3) to (2,2);
		\draw (1,2) node[fill=blue] {};
		\draw (1,3) node[fill=blue] {};
		\draw (2,3) node[fill=blue] {};
		\draw (1,1) node[draw=black,fill=black!50] {};
		\draw (2,2) node[draw=black,fill=black!50] {};
		\draw (3,3) node[draw=black,fill=black!50] {};
		\coordinate (p) at ([yshift=1ex]current bounding box.center);
	}\right)
	&=\topo\left(\tikz[baseline=(p)]{
		\tikzset{scale=0.500000,yscale=-1,shorten >=2pt,decoration={markings,mark=at position 0.5 with {\arrow{>}}}}
		\tikzset{every node/.style={circle,draw,solid,fill=black!50,inner sep=0pt,minimum width=4pt}}
		\draw[postaction={decorate}] (2,2) to (1,1);
		\draw[postaction={decorate}] (3,3) to (2,2);
		\draw[postaction={decorate}]  (2,2) to (1,2);
		\draw[postaction={decorate},bend right=30]  (3,3) to (1,3);
		\draw[postaction={decorate}]  (3,3) to (2,3);
		\draw (1,2) node[fill=blue] {};
		\draw (1,3) node[fill=blue] {};
		\draw (2,3) node[fill=blue] {};
		\draw (1,1) node[draw=black,fill=black!50] {};
		\draw (2,2) node[draw=black,fill=black!50] {};
		\draw (3,3) node[draw=black,fill=black!50] {};
		\coordinate (p) at ([yshift=1ex]current bounding box.center);
	}\right)-\topo\left(\tikz[baseline=(p)]{
		\tikzset{scale=0.500000,yscale=-1,shorten >=2pt,decoration={markings,mark=at position 0.5 with {\arrow{>}}}}
		\tikzset{every node/.style={circle,draw,solid,fill=black!50,inner sep=0pt,minimum width=4pt}}
		\draw[postaction={decorate}] (2,2) to (1,1);
		\draw[postaction={decorate}] (3,3) to (2,2);
		\draw[postaction={decorate}]  (1,1) to (1,2);
		\draw[postaction={decorate},bend right=30]  (3,3) to (1,3);
		\draw[postaction={decorate}]  (3,3) to (2,3);
		\draw (1,2) node[fill=blue] {};
		\draw (1,3) node[fill=blue] {};
		\draw (2,3) node[fill=blue] {};
		\draw (1,1) node[draw=black,fill=black!50] {};
		\draw (2,2) node[draw=black,fill=black!50] {};
		\draw (3,3) node[draw=black,fill=black!50] {};
		\coordinate (p) at ([yshift=1ex]current bounding box.center);
	}\right)
	-\topo\left(\tikz[baseline=(p)]{
		\tikzset{scale=0.500000,yscale=-1,shorten >=2pt,decoration={markings,mark=at position 0.5 with {\arrow{>}}}}
		\tikzset{every node/.style={circle,draw,solid,fill=black!50,inner sep=0pt,minimum width=4pt}}
		\draw[postaction={decorate}] (2,2) to (1,1);
		\draw[postaction={decorate}] (3,3) to (2,2);
		\draw[postaction={decorate}]  (2,2) to (1,2);
		\draw[postaction={decorate},bend left=30]  (1,1) to (1,3);
		\draw[postaction={decorate}]  (3,3) to (2,3);
		\draw (1,2) node[fill=blue] {};
		\draw (1,3) node[fill=blue] {};
		\draw (2,3) node[fill=blue] {};
		\draw (1,1) node[draw=black,fill=black!50] {};
		\draw (2,2) node[draw=black,fill=black!50] {};
		\draw (3,3) node[draw=black,fill=black!50] {};
		\coordinate (p) at ([yshift=1ex]current bounding box.center);
	}
	\right)+\topo\left(\tikz[baseline=(p)]{
		\tikzset{scale=0.500000,yscale=-1,shorten >=2pt,decoration={markings,mark=at position 0.5 with {\arrow{>}}}}
		\tikzset{every node/.style={circle,draw,solid,fill=black!50,inner sep=0pt,minimum width=4pt}}
		\draw[postaction={decorate}] (2,2) to (1,1);
		\draw[postaction={decorate}] (3,3) to (2,2);
		\draw[postaction={decorate}]  (1,1) to (1,2);
		\draw[postaction={decorate},bend left=30]  (1,1) to (1,3);
		\draw[postaction={decorate}]  (3,3) to (2,3);
		\draw (1,2) node[fill=blue] {};
		\draw (1,3) node[fill=blue] {};
		\draw (2,3) node[fill=blue] {};
		\draw (1,1) node[draw=black,fill=black!50] {};
		\draw (2,2) node[draw=black,fill=black!50] {};
		\draw (3,3) node[draw=black,fill=black!50] {};
		\coordinate (p) at ([yshift=1ex]current bounding box.center);
	}
	\right)\\
	&-\topo\left(\tikz[baseline=(p)]{
		\tikzset{scale=0.500000,yscale=-1,shorten >=2pt,decoration={markings,mark=at position 0.5 with {\arrow{>}}}}
		\tikzset{every node/.style={circle,draw,solid,fill=black!50,inner sep=0pt,minimum width=4pt}}
		\draw[postaction={decorate}] (2,2) to (1,1);
		\draw[postaction={decorate}] (3,3) to (2,2);
		\draw[postaction={decorate}]  (2,2) to (1,2);
		\draw[postaction={decorate},bend right=30]  (3,3) to (1,3);
		\draw[postaction={decorate}]  (2,2) to (2,3);
		\draw (1,2) node[fill=blue] {};
		\draw (1,3) node[fill=blue] {};
		\draw (2,3) node[fill=blue] {};
		\draw (1,1) node[draw=black,fill=black!50] {};
		\draw (2,2) node[draw=black,fill=black!50] {};
		\draw (3,3) node[draw=black,fill=black!50] {};
		\coordinate (p) at ([yshift=1ex]current bounding box.center);
	}
	\right)+\topo\left(\tikz[baseline=(p)]{
		\tikzset{scale=0.500000,yscale=-1,shorten >=2pt,decoration={markings,mark=at position 0.5 with {\arrow{>}}}}
		\tikzset{every node/.style={circle,draw,solid,fill=black!50,inner sep=0pt,minimum width=4pt}}
		\draw[postaction={decorate}] (2,2) to (1,1);
		\draw[postaction={decorate}] (3,3) to (2,2);
		\draw[postaction={decorate}]  (1,1) to (1,2);
		\draw[postaction={decorate},bend right=30]  (3,3) to (1,3);
		\draw[postaction={decorate}]  (2,2) to (2,3);
		\draw (1,2) node[fill=blue] {};
		\draw (1,3) node[fill=blue] {};
		\draw (2,3) node[fill=blue] {};
		\draw (1,1) node[draw=black,fill=black!50] {};
		\draw (2,2) node[draw=black,fill=black!50] {};
		\draw (3,3) node[draw=black,fill=black!50] {};
		\coordinate (p) at ([yshift=1ex]current bounding box.center);
	}
	\right)
	+\topo\left(\tikz[baseline=(p)]{
		\tikzset{scale=0.500000,yscale=-1,shorten >=2pt,decoration={markings,mark=at position 0.5 with {\arrow{>}}}}
		\tikzset{every node/.style={circle,draw,solid,fill=black!50,inner sep=0pt,minimum width=4pt}}
		\draw[postaction={decorate}] (2,2) to (1,1);
		\draw[postaction={decorate}] (3,3) to (2,2);
		\draw[postaction={decorate}]  (2,2) to (1,2);
		\draw[postaction={decorate},bend left=30]  (1,1) to (1,3);
		\draw[postaction={decorate}]  (2,2) to (2,3);
		\draw (1,2) node[fill=blue] {};
		\draw (1,3) node[fill=blue] {};
		\draw (2,3) node[fill=blue] {};
		\draw (1,1) node[draw=black,fill=black!50] {};
		\draw (2,2) node[draw=black,fill=black!50] {};
		\draw (3,3) node[draw=black,fill=black!50] {};
		\coordinate (p) at ([yshift=1ex]current bounding box.center);
	}
	\right)-\topo\left(\tikz[baseline=(p)]{
		\tikzset{scale=0.500000,yscale=-1,shorten >=2pt,decoration={markings,mark=at position 0.5 with {\arrow{>}}}}
		\tikzset{every node/.style={circle,draw,solid,fill=black!50,inner sep=0pt,minimum width=4pt}}
		\draw[postaction={decorate}] (2,2) to (1,1);
		\draw[postaction={decorate}] (3,3) to (2,2);
		\draw[postaction={decorate}]  (1,1) to (1,2);
		\draw[postaction={decorate},bend left=30]  (1,1) to (1,3);
		\draw[postaction={decorate}]  (2,2) to (2,3);
		\draw (1,2) node[fill=blue] {};
		\draw (1,3) node[fill=blue] {};
		\draw (2,3) node[fill=blue] {};
		\draw (1,1) node[draw=black,fill=black!50] {};
		\draw (2,2) node[draw=black,fill=black!50] {};
		\draw (3,3) node[draw=black,fill=black!50] {};
		\coordinate (p) at ([yshift=1ex]current bounding box.center);
	}
	\right)
	\end{aligned}$$
	Note, that the sign equals $-1$ to the power of the number of vertical arrows.
\end{example}

\subsection{Weight preserving sijections}
\begin{dfn}
	\label{defPhi}
	Fix integers $n \le N \in \mathbb{N}$ and let $X=(x_1,\ldots,x_n)$ be a tuple of formal symbols. Let $S$ be a weighted signed set, in which all elements have weights of the form~$\pm \prod_{i\in [n]} x_i^{c_i}$ with $\sum c_i = N-n$ constant, and let $P=(p_1,\ldots,p_n)$ with all $p_i \in [N]$ distinct.
	
	We define an unweighted signed set:
	\begin{align*}
	\phi(S,X,P) := \{(s,L) \colon &s \in S, w(s)= \pm \prod_{i\in [n]} x_i^{c_i}, L=(L_1,\ldots,L_n), L_i=(\ell_{i,1},\ldots,\ell_{i,c_i}),\\
	&\text{all }\ell_{i,j}\text{ and }p_i\text{ distinct integers in }[N], \forall \ell_{i,j} \in L_i\ \ell_{i,j} < p_i\}
	\end{align*}
	with weight $w((s,L)) = \mathrm{sign}(w_S(s))$. 
\end{dfn}
\begin{rem}
	Note that every integer in $[N]$ appears either in $P$ or in some $L_i$. Also note that we allow $c_i = 0$, in which case the corresponding tuple $L_i$ is an empty tuple, which we write as $()$. Furthermore we have
	$$\phi(A,X,P) + \phi(B,X,P) = \phi(A+B,X,P)$$
	and for an unweighted signed set $C$ there is a bijection
	$$C\times \phi(A,X,P) \leftrightarrow \phi(C\times A,X,P).$$
	Finally, if we have a weight-preserving sijection $\psi$ between $A$ and $B$, then there is a sijection between $\phi(A,X,P)$ and $\phi(B,X,P)$. This holds as weight-preservation corresponds to the lists $L_i$ having the same lengths.
\end{rem}

Now, we can formalise Example \ref{exampleTrick} as follows:
\begin{example}
	\label{exampleTrick2}
	We again let $\alpha=(1,1,1)$. And here we set $p_1=1, p_2=4, p_3=6$.
	Then \begin{align*}
	&\{f \in \Topo\left (\tikz[baseline=(p)]{
		\tikzset{scale=0.500000,yscale=-1,shorten >=2pt,decoration={markings,mark=at position 0.5 with {\arrow{>}}}}
		\tikzset{every node/.style={circle,draw,solid,fill=black!50,inner sep=0pt,minimum width=4pt}}
		\draw[postaction={decorate}] (1,1) to (0,0);
		\draw[postaction={decorate}] (2,2) to (1,1);	
		\draw[postaction={decorate}]  (2,2) to (1,2);
		\draw[postaction={decorate}]  (1,2) to (1,1);
		\draw[postaction={decorate},bend right=30]  (2,2) to (0,2);
		\draw[postaction={decorate},bend right=30]  (0,2) to (0,0);
		\draw[postaction={decorate}]  (1,1) to (0,1);
		\draw[postaction={decorate}]  (0,1) to (0,0);
		\draw (1,2) node[fill=blue] {};
		\draw (0,2) node[fill=blue] {};
		\draw (0,1) node[fill=blue] {};
		\draw (0,0) node[draw=black,fill=black!50] {};
		\draw (1,1) node[draw=black,fill=black!50] {};
		\draw (2,2) node[draw=black,fill=black!50] {};
		\coordinate (p) at ([yshift=1ex]current bounding box.center);
	}\right)\colon f(u_i)=p_i\}\\
	&\leftrightarrow \{f \in \Topo\left (\tikz[baseline=(p)]{
		\tikzset{scale=0.500000,yscale=-1,shorten >=2pt,decoration={markings,mark=at position 0.5 with {\arrow{>}}}}
		\tikzset{every node/.style={circle,draw,solid,fill=black!50,inner sep=0pt,minimum width=4pt}}
		\draw[postaction={decorate}] (1,1) to (0,0);
		\draw[postaction={decorate}] (2,2) to (1,1);	
		\draw[postaction={decorate}]  (2,2) to (1,2);
		\draw[postaction={decorate},bend right=30]  (2,2) to (0,2);
		\draw[postaction={decorate}]  (1,1) to (0,1);
		\draw (1,2) node[fill=blue] {};
		\draw (0,2) node[fill=blue] {};
		\draw (0,1) node[fill=blue] {};
		\draw (0,0) node[draw=black,fill=black!50] {};
		\draw (1,1) node[draw=black,fill=black!50] {};
		\draw (2,2) node[draw=black,fill=black!50] {};
		\coordinate (p) at ([yshift=1ex]current bounding box.center);
	}\right)\colon f(u_i)=p_i\} - \{f \in \Topo\left (\tikz[baseline=(p)]{
		\tikzset{scale=0.500000,yscale=-1,shorten >=2pt,decoration={markings,mark=at position 0.5 with {\arrow{>}}}}
		\tikzset{every node/.style={circle,draw,solid,fill=black!50,inner sep=0pt,minimum width=4pt}}
		\draw[postaction={decorate}] (1,1) to (0,0);
		\draw[postaction={decorate}] (2,2) to (1,1);
		\draw[postaction={decorate}]  (1,1) to (1,2);
		\draw[postaction={decorate},bend right=30]  (2,2) to (0,2);
		\draw[postaction={decorate}]  (1,1) to (0,1);
		\draw (1,2) node[fill=blue] {};
		\draw (0,2) node[fill=blue] {};
		\draw (0,1) node[fill=blue] {};
		\draw (0,0) node[draw=black,fill=black!50] {};
		\draw (1,1) node[draw=black,fill=black!50] {};
		\draw (2,2) node[draw=black,fill=black!50] {};
		\coordinate (p) at ([yshift=1ex]current bounding box.center);
	}\right)\colon f(u_i)=p_i\}\\
	&= \left \{\tikz[baseline=(p)]{
		\tikzset{scale=0.500000,yscale=-1,shorten >=2pt,decoration={markings,mark=at position 0.5 with {\arrow{>}}}}
		\tikzset{every node/.style={circle,draw,solid,fill=black!50,inner sep=0pt,minimum width=4pt}}
		\draw[postaction={decorate}] (1,1) to (0,0);
		\draw[postaction={decorate}] (2,2) to (1,1);	
		\draw[postaction={decorate}]  (2,2) to (1,2);
		\draw[postaction={decorate},bend right=30]  (2,2) to (0,2);
		\draw[postaction={decorate}]  (1,1) to (0,1);
		\draw (0,0) node[draw=black,fill=black!50,label={[label distance=0.05cm]45:$1$}] {};
		\draw (0,1) node[fill=blue,label={[label distance=0.05cm]180:$2$}] {};
		\draw (1,1) node[draw=black,fill=black!50,label={[label distance=0.05cm]45:$4$}] {};
		\draw (0,2) node[fill=blue,label={[label distance=0.05cm]225:$3$}] {};
		\draw (1,2) node[fill=blue,label={[label distance=0.1cm]270:$5$}] {};
		\draw (2,2) node[draw=black,fill=black!50,label={[label distance=0.05cm]45:$6$}] {};
		\coordinate (p) at ([yshift=1ex]current bounding box.center);
	},\tikz[baseline=(p)]{
		\tikzset{scale=0.500000,yscale=-1,shorten >=2pt,decoration={markings,mark=at position 0.5 with {\arrow{>}}}}
		\tikzset{every node/.style={circle,draw,solid,fill=black!50,inner sep=0pt,minimum width=4pt}}
		\draw[postaction={decorate}] (1,1) to (0,0);
		\draw[postaction={decorate}] (2,2) to (1,1);	
		\draw[postaction={decorate}]  (2,2) to (1,2);
		\draw[postaction={decorate},bend right=30]  (2,2) to (0,2);
		\draw[postaction={decorate}]  (1,1) to (0,1);
		\draw (0,0) node[draw=black,fill=black!50,label={[label distance=0.05cm]45:$1$}] {};
		\draw (0,1) node[fill=blue,label={[label distance=0.05cm]180:$2$}] {};
		\draw (1,1) node[draw=black,fill=black!50,label={[label distance=0.05cm]45:$4$}] {};
		\draw (0,2) node[fill=blue,label={[label distance=0.05cm]225:$5$}] {};
		\draw (1,2) node[fill=blue,label={[label distance=0.1cm]270:$3$}] {};
		\draw (2,2) node[draw=black,fill=black!50,label={[label distance=0.05cm]45:$6$}] {};
		\coordinate (p) at ([yshift=1ex]current bounding box.center);
	},\tikz[baseline=(p)]{
		\tikzset{scale=0.500000,yscale=-1,shorten >=2pt,decoration={markings,mark=at position 0.5 with {\arrow{>}}}}
		\tikzset{every node/.style={circle,draw,solid,fill=black!50,inner sep=0pt,minimum width=4pt}}
		\draw[postaction={decorate}] (1,1) to (0,0);
		\draw[postaction={decorate}] (2,2) to (1,1);	
		\draw[postaction={decorate}]  (2,2) to (1,2);
		\draw[postaction={decorate},bend right=30]  (2,2) to (0,2);
		\draw[postaction={decorate}]  (1,1) to (0,1);
		\draw (0,0) node[draw=black,fill=black!50,label={[label distance=0.05cm]45:$1$}] {};
		\draw (0,1) node[fill=blue,label={[label distance=0.05cm]180:$3$}] {};
		\draw (1,1) node[draw=black,fill=black!50,label={[label distance=0.05cm]45:$4$}] {};
		\draw (0,2) node[fill=blue,label={[label distance=0.05cm]225:$2$}] {};
		\draw (1,2) node[fill=blue,label={[label distance=0.1cm]270:$5$}] {};
		\draw (2,2) node[draw=black,fill=black!50,label={[label distance=0.05cm]45:$6$}] {};
		\coordinate (p) at ([yshift=1ex]current bounding box.center);
	},\tikz[baseline=(p)]{
		\tikzset{scale=0.500000,yscale=-1,shorten >=2pt,decoration={markings,mark=at position 0.5 with {\arrow{>}}}}
		\tikzset{every node/.style={circle,draw,solid,fill=black!50,inner sep=0pt,minimum width=4pt}}
		\draw[postaction={decorate}] (1,1) to (0,0);
		\draw[postaction={decorate}] (2,2) to (1,1);	
		\draw[postaction={decorate}]  (2,2) to (1,2);
		\draw[postaction={decorate},bend right=30]  (2,2) to (0,2);
		\draw[postaction={decorate}]  (1,1) to (0,1);
		\draw (0,0) node[draw=black,fill=black!50,label={[label distance=0.05cm]45:$1$}] {};
		\draw (0,1) node[fill=blue,label={[label distance=0.05cm]180:$3$}] {};
		\draw (1,1) node[draw=black,fill=black!50,label={[label distance=0.05cm]45:$4$}] {};
		\draw (0,2) node[fill=blue,label={[label distance=0.05cm]225:$5$}] {};
		\draw (1,2) node[fill=blue,label={[label distance=0.1cm]270:$2$}] {};
		\draw (2,2) node[draw=black,fill=black!50,label={[label distance=0.05cm]45:$6$}] {};
		\coordinate (p) at ([yshift=1ex]current bounding box.center);
	}
	\right \}-\left\{ \tikz[baseline=(p)]{
		\tikzset{scale=0.500000,yscale=-1,shorten >=2pt,decoration={markings,mark=at position 0.5 with {\arrow{>}}}}
		\tikzset{every node/.style={circle,draw,solid,fill=black!50,inner sep=0pt,minimum width=4pt}}
		\draw[postaction={decorate}] (1,1) to (0,0);
		\draw[postaction={decorate}] (2,2) to (1,1);	
		\draw[postaction={decorate}]  (1,1) to (1,2);
		\draw[postaction={decorate},bend right=30]  (2,2) to (0,2);
		\draw[postaction={decorate}]  (1,1) to (0,1);
		\draw (0,0) node[draw=black,fill=black!50,label={[label distance=0.05cm]45:$1$}] {};
		\draw (0,1) node[fill=blue,label={[label distance=0.05cm]180:$2$}] {};
		\draw (1,1) node[draw=black,fill=black!50,label={[label distance=0.05cm]45:$4$}] {};
		\draw (0,2) node[fill=blue,label={[label distance=0.05cm]225:$5$}] {};
		\draw (1,2) node[fill=blue,label={[label distance=0.1cm]270:$3$}] {};
		\draw (2,2) node[draw=black,fill=black!50,label={[label distance=0.05cm]45:$6$}] {};
		\coordinate (p) at ([yshift=1ex]current bounding box.center);
	}, \tikz[baseline=(p)]{
		\tikzset{scale=0.500000,yscale=-1,shorten >=2pt,decoration={markings,mark=at position 0.5 with {\arrow{>}}}}
		\tikzset{every node/.style={circle,draw,solid,fill=black!50,inner sep=0pt,minimum width=4pt}}
		\draw[postaction={decorate}] (1,1) to (0,0);
		\draw[postaction={decorate}] (2,2) to (1,1);	
		\draw[postaction={decorate}]  (1,1) to (1,2);
		\draw[postaction={decorate},bend right=30]  (2,2) to (0,2);
		\draw[postaction={decorate}]  (1,1) to (0,1);
		\draw (0,0) node[draw=black,fill=black!50,label={[label distance=0.05cm]45:$1$}] {};
		\draw (0,1) node[fill=blue,label={[label distance=0.05cm]180:$3$}] {};
		\draw (1,1) node[draw=black,fill=black!50,label={[label distance=0.05cm]45:$4$}] {};
		\draw (0,2) node[fill=blue,label={[label distance=0.05cm]225:$5$}] {};
		\draw (1,2) node[fill=blue,label={[label distance=0.1cm]270:$2$}] {};
		\draw (2,2) node[draw=black,fill=black!50,label={[label distance=0.05cm]45:$6$}] {};
		\coordinate (p) at ([yshift=1ex]current bounding box.center);
	} \right\}\\
	&\leftrightarrow \left\{ (x_2 x_3^2, ((),(2),(3,5))),(x_2 x_3^2, ((),(2),(5,3))),(x_2 x_3^2, ((),(3),(2,5))),(x_2 x_3^2, ((),(3),(5,2))) \right\} \\
	&\quad\quad- \left\{ (x_2^2 x_3, ((),(2,3),(5))),(x_2^2 x_3, ((),(3,2),(5))) \right\}\\
	&=  \phi(x_2 x_3^2, (x_1,x_2,x_3),(1,4,6)) - \phi(x_2^2 x_3, (x_1,x_2,x_3),(1,4,6))\\
	&= \phi((x_2-x_1)(x_3-x_1)(x_3-x_2), (x_1,x_2,x_3),(1,4,6)),
	\end{align*}	
	where we have omitted the sets corresponding to vertical arrows down from $u_1$, because these sets are empty.
\end{example}
More generally we have:
\begin{lemma}
	\label{lemmaTrick}
	Fix $\alpha =(\alpha_1,\ldots,\alpha_n)$ with $\alpha_i \ge 1$ and let $X=(x_1,\ldots,x_n)$ a tuple of formal variables. Then for any tuple $P=(p_1,\ldots,p_n)$ with $p_1<\ldots<p_n$ and $p_i \in [n+\sum_{1\le i < j \le n} \alpha_i \alpha_j]$ we have
	\begin{align*}
	\{f \in \Topo(G_X(\alpha))\colon f(u_i)=p_i\}
	\leftrightarrow \phi(\prod_{1\le i<j\le n} \{x_j,-x_i\}^{\alpha_i \alpha_j},(x_1,\ldots,x_n),P)
	\end{align*}
\end{lemma}
\begin{proof}
	We apply Lemma \ref{lemmaTrickOne} on all vertices $w_{i,j,k}$, as in Example \ref{exampleTrick2}. 
	For each vertex $w_{i,j,k}$ we can choose between a horizontal arrow from $u_j$ or a vertical arrow from $u_i$. The sign equals $-1$ to the power of times we chose the vertical arrow.
	
	Now, for each term in this expansion, we construct the tuples $L_p$ as follows:
	We iterate over all vertices $w_{i,j,k}$ sorted by their indices in lexicographical order. If $u_p \rightarrow w_{i,j,k}$, we add $f(w_{i,j,k})$ to the tuple $L_p$. This way the lists are guaranteed to have the correct length, and we can undo the process by iterating over all vertices $w_{i,j,k}$ in the same order. 
\end{proof}

\begin{lemma}
	\label{lemmaPhiIntegration}
	Fix $1\le j \ne k <n$ and $p_1,\ldots,p_{j-1},p_{j+1},\ldots,p_n$ and $X=(x_1,\ldots,x_{n})$.
	
	Let $S$ be a signed set for which all weights are of the form $\pm \prod_{i\in [n]} x_i^{c_i}$ with $c_j=:\ell$.
	We have a sijection
	
	\begin{align*}
	&[\ell+1] \times \sum_{p_j<p_{k}} \phi(S,X,(p_1,p_2,\ldots,p_n))\\
	&\leftrightarrow \phi(S',(x_1,\ldots,x_{j-1},x_{j+1},\ldots,x_{n}),(p_1,\ldots,p_{j-1},p_{j+1},\ldots,p_n))
	\end{align*}
	where $S'$ is a copy of $S$, for which $x_j^\ell$ is replaced by $x_{k}^{\ell+1}$ in all weights.
\end{lemma}
\begin{proof}
	Consider an element $(x,(s,L))\in [\ell+1] \times \phi(S,X,(p_1,p_2,\ldots,p_n))$ for any $p_j$.
	Note that the tuple $L_j$ contains exactly $\ell$ elements. Insert $p_j$ into this tuple at position $x$ and append the resulting tuple to $L_{k}$. Since $p_j$ is the larger than all elements in $L_j$, this process is reversible. 
\end{proof}

\begin{lemma}
	\label{lemmaPhiIntegration2}
	Fix $p_1 < \ldots < p_{j-1} < p_{j+1} < \ldots < p_n$.
	We have a sijection
	\begin{align*}
	&\sum_{p_{j-1}<p_j<p_{j+1}} \phi(S,X,(p_1,p_2,\ldots,p_n))\\
	&\leftrightarrow \sum_{p_j<p_{j+1}} \phi(S,X,(p_1,p_2,\ldots,p_n))-\sum_{p_j<p_{j-1}} \phi(S,X,(p_1,p_2,\ldots,p_n))
	\end{align*}
\end{lemma}
\begin{proof}
	Every choice of $p_j < p_{j+1}$ either satisfies $p_{j-1} <p_j<p_{j+1}$ or $p_{j} < p_{j-1}$. The result follows.
\end{proof}

\subsection{Matrices.}
We start this subsection with a definition of a signed set $D$:

\begin{dfn}
	\label{defMatrixD}
	Let $M$ be an $n \times n$ matrix, where each $M_{i,j}$ is any signed set with weight function $w_{M_{i,j}}$. Define a signed set
	$$D(M):=\{ (\sigma, (m_1,\ldots,m_n)) : \sigma \in S_n, m_i \in M_{i,\sigma(i)}\}$$
	together with the weight function
	$$w_{D(M)}(\sigma,(m_1,\ldots,m_n)) := \mathrm{sgn}(\sigma) \cdot \prod_i w_{M_{i,\sigma(i)}}(m_i).$$
\end{dfn}
Note, that $w(D(M))=\det(M')$, where $M'$ is the matrix with entries $M'_{i,j}=w(M_{i,j})$.

Using Definition \ref{defMatrixD}, one could rewrite Gessels proof of the Vandermonde Determinant as a weight-preserving sijection between the signed sets $\prod_{1\le i<j\le n} \{x_j,-x_i\}$ and $D \left(\begin{pmatrix}
1&x_1&\ldots&x^{n-1}\\
\vdots&&&\vdots\\
1&x_n&\ldots&x_n^{n-1}
\end{pmatrix} \right)$.

For the rest of this subsection we extend this to a weight-preserving sijection for the generalised Vandermonde Determinant. For usual matrices, we know that we can add or subtract a row from another without changing the determinant. We can express this fact for $D(M)$ using a sijection as follows:

\newcommand*{\horzbar}{\rule[.5ex]{2.5ex}{0.5pt}}
\begin{lemma}
	\label{RowSubtractLemma}
	Fix any $i \ne j \in [n]$. Let $A$ be an $n \times n$ matrix
	and $A'$ be the result after adding (resp. subtracting) the $i$th row of $A$ from (resp. to) the $j$th row. In other words we have
	$$(A')_{s,t} = \begin{cases}A_{s,t},&\text{if }s \ne j\\A_{j,t} \pm A_{i,t},&\text{if }s = j\end{cases}$$
	where addition or subtraction of signed sets is as defined above.
	We have a sijection $$D(A) \leftrightarrow D(A').$$
\end{lemma}
\begin{proof}
	We have $D(A) \subset D(A')$, as $A_{j,t} \subset A_{j,t} \pm A_{i,t}$, so we only need to find a sign-reversing involution $\phi$ on $D(A') \setminus D(A)$. Take any $x=(\sigma,(m_1,\ldots,m_n)) \in D(A') \setminus D(A)$, so we have $m_i \in A_{i,\sigma(i)}$ and $m_j \in A_{i,\sigma(j)}$. Now define $\phi(x):= (\sigma',(m_1',\ldots,m_n'))$ with $\sigma' := \sigma(i\ j) $ and $$m_k' := \begin{cases}m_j,&\text{if }k=i\\m_i,&\text{if }k=j\\m_k,&\text{otherwise.}\end{cases}$$
	We have $m_i' = m_j \in A_{i,\sigma(j)} = A_{i,\sigma'(i)}$ and $m_j' = m_i \in A_{i,\sigma(i)} = A_{i,\sigma'(j)}$, so $\phi(x) \in D(A') \setminus D(A)$. Additionally we have $\mathrm{sgn}(\sigma') = -\mathrm{sgn}(\sigma)$, i.e. $\phi$ is sign-reversing. Also, $\phi$ is clearly an involution.
\end{proof}

%
%

\begin{dfn}
	Let $\alpha,m \ge 0$ and $y_1,\ldots,y_k,x$ be formal variables.
	Define $M$ as the $\alpha \times m$ matrix with entries:
	$$M((y_1,\ldots,y_k),x,\alpha,m)_{i,j} = B((y_1,\ldots,y_k,\underbrace{x,\ldots,x}_{i\text{ times}}),j-i).$$
\end{dfn}
\begin{dfn}
	Let $\alpha=(\alpha_1,\ldots,\alpha_n)$ with $\alpha_i \ge 0$, $m = \sum \alpha_i$ and $X=(x_1,\ldots,x_n)$ and $Y=(y_1,\ldots,y_k)$ tuples of formal variables.
	
	Define the following $m \times m$ block matrix:
	$$H(Y, X, \alpha) = \begin{pmatrix}M(Y,x_1,\alpha_1,m)\\\vdots\\M(Y,x_n,\alpha_n,m)\end{pmatrix}.$$

\end{dfn}

Before stating the next Theorem in its general form, we would like to show how to use Lemmas \ref{RowSubtractLemma} and \ref{lemmaBinomial2} with an example.
\begin{example}
	\label{exampleDoublyGeneralisedVandermondeStep}
	Let $n=2, \alpha_1=2, \alpha_2 = 2$.
	Then $$D(H((),(x_1,x_2),(2,2)))=D\left(\begin{pmatrix}
	B((x_1),0)&B((x_1),1)&B((x_1),2)&B((x_1),3)\\
	\emptyset&B((x_1,x_1),0)&B((x_1,x_1),1)&B((x_1,x_1),2)\\
	B((x_2),0)&B((x_2),1)&B((x_2),2)&B((x_2),3)\\
	\emptyset&B((x_2,x_2),0)&B((x_2,x_2),1)&B((x_2,x_2),2)
	\end{pmatrix}\right).$$
	Now we apply Lemma \ref{RowSubtractLemma} by subtracting row $1$ from row $3$.
	Using Lemma \ref{lemmaBinomial2} the entry at position $(3,j)$ then becomes
	$$B((x_2),j-1)-B((x_1),j-1)\leftrightarrow \{x_2,-x_1\} \times B((x_1, x_2),j-2)$$
	We can pull out the common factor $\{x_2,-x_1\}$ out of row $3$.
	$$\leftrightarrow \{x_2,-x_1\} \times D\left(\begin{pmatrix}
	B((x_1),0)&B((x_1),1)&B((x_1),2)&B((x_1),3)\\
	\emptyset&B((x_1,x_1),0)&B((x_1,x_1),1)&B((x_1,x_1),2)\\
	\emptyset&B((x_1,x_2),0)&B((x_1,x_2),1)&B((x_1,x_2),2)\\
	\emptyset&B((x_2,x_2),0)&B((x_2,x_2),1)&B((x_2,x_2),2)
	\end{pmatrix}\right).$$
	Now we apply Lemma \ref{RowSubtractLemma} by subtracting row $3$ from row $4$. The entry at position $(4,j)$ then becomes
	$$B((x_2,x_2),j-2)-B((x_1,x_2),j-2)\leftrightarrow \{x_2,-x_1\} \times B((x_1,x_2,x_2),j-3)$$
	
	We can pull out the common factor $\{x_2,-x_1\}$ out of row $4$.
	$$\leftrightarrow \{x_2,-x_1\}^2 \times D\left(\begin{pmatrix}
	B((x_1),0)&B((x_1),1)&B((x_1),2)&B((x_1),3)\\
	\emptyset&B((x_1,x_1),0)&B((x_1,x_1),1)&B((x_1,x_1),2)\\
	\emptyset&B((x_1,x_2),0)&B((x_1,x_2),1)&B((x_1,x_2),2)\\
	\emptyset&\emptyset&B((x_1,x_2,x_2),0)&B((x_1,x_2,x_2),1)
	\end{pmatrix}\right).$$
	Now all entries in the left-most column are empty sets, except the top-left corner, which contains a single element with weight $1$. Hence we have a sijection
	$$\leftrightarrow \{x_2,-x_1\}^2 \times D\left(\begin{pmatrix}
	B((x_1,x_1),0)&B((x_1,x_1),1)&B((x_1,x_1),2)\\
	B((x_1,x_2),0)&B((x_1,x_2),1)&B((x_1,x_2),2)\\
	\emptyset&B((x_1,x_2,x_2),0)&B((x_1,x_2,x_2),1)
	\end{pmatrix}\right).$$
	We have shown:
	$$D(H((),(x_1,x_2),(2,2))) \leftrightarrow \{x_2,-x_1\}^2 \times D\left(H((x_1),(x_1,x_2),(1,2))\right).$$
	
\end{example}
\begin{lemma}
	\label{DoublyGeneralisedVandermondeStep}
	Fix $\alpha=(\alpha_1,\ldots,\alpha_n)$ with $\alpha_i \ge 0$, $m = \sum \alpha_i$ and $X=(x_1,\ldots,x_n)$ and $Y=(y_1,\ldots,y_k)$ tuples of formal variables.
	\begin{itemize}
		\item[a)] For $\alpha_1 \ge 1$ there exists a sijection 
		\begin{align*}
		&D(H(Y, X, \alpha))\\
		&\leftrightarrow \prod_{2 \le j \le n} \{x_j,-x_1\}^{\alpha_j} \times D(H((y_1,\ldots,y_k,x_1),X,( \alpha_1 - 1,\alpha_2,\ldots,\alpha_n))).
		\end{align*}
		\item[b)] For $\alpha_1 = 0$ we have
		$$D(H(Y, X, \alpha)) = D(H(Y,(x_2,\ldots,x_n), (\alpha_2,\ldots,\alpha_n)))$$
	\end{itemize}
\end{lemma}
\begin{proof}
	For part a) we do the same as in Example \ref{exampleDoublyGeneralisedVandermondeStep}. We start with the matrix $D(H(Y, X, \alpha))$ with $\alpha_1 \ge 1$. We apply Lemmas \ref{RowSubtractLemma} and \ref{lemmaBinomial2} multiple times by doing row operations as follows:
	
	We subtract row $1$ from row $\alpha_1 + \ldots  + \alpha_{j-1} + 1$ for all $j \in \{2,\ldots,n\}$ with $\alpha_j \ge 1$. Then for all $q \in \{2,\ldots,\alpha_j\}$ we subtract row  $\alpha_1 + \ldots  + \alpha_{j-1} + q-1$ from row  $\alpha_1 + \ldots  + \alpha_{j-1} + q$. Each time we pull out the factor $\{x_j,-x_1\}$. 
	
	In the resulting matrix the left-most column only has empty sets, except for the top-left corner, which contains a single element with weight $1$. We can therefore delete the first row and column and are left with the matrix $D(H((y_1,\ldots,y_k,x_1),X,( \alpha_1 - 1,\alpha_2,\ldots,\alpha_n)))$.
	
	Part b) follows directly from the definition of $H$: The top-most block simply has height $0$.
\end{proof}
\begin{thm}
	\label{DoublyGeneralisedVandemonde}
	Fix $\alpha=(\alpha_1,\ldots,\alpha_n)$ with $\alpha_i \ge 0$, $m = \sum \alpha_i$ and $X=(x_1,\ldots,x_n)$ and $Y=(y_1,\ldots,y_k)$ tuples of formal variables.
	
	Then there exists a sijection 
	$$D(H(Y, X,\alpha)) \leftrightarrow \prod_{1 \le i < j \le n} \{x_j,-x_i\}^{\alpha_i \alpha_j} .$$
	
\end{thm}
\begin{proof}
	Proof by induction on $n$.
	For $n=0$, $D(H(Y,(),())) = \{1\}$, the statement is true.
	Now suppose $n\ge 1$. We apply Lemma \ref{DoublyGeneralisedVandermondeStep} part a) until $\alpha_1$ is decreased down to $0$. In total we have pulled out a factor $\prod_{2 \le j \le n} \{x_j,-x_1\}^{\alpha_1 \alpha_j}$. Then we remove $\alpha_1$ and $x_1$ by applying Lemma \ref{DoublyGeneralisedVandermondeStep} b) and by the induction hypothesis, the result follows.
\end{proof}
For $Y=()$ this theorem is a sijective version of Theorem \ref{generalVandermonde}. In Appendix A, you can find a further generalisation of this theorem, but it is not relevant to our proof.


\subsection{Putting it all together.}

We need to define some more matrices
\begin{dfn}
	Let $\alpha,m \ge 0$ and $x$ be formal variables.
	Define $M_1$ as the $\alpha \times m$ matrix with entries:
	$$M_1(x,\alpha,m)_{i,j} = B((\underbrace{x,\ldots,x}_{i+1\text{ times}}),j-i).$$
	
	Define $M_2$ as the $1 \times m$ matrix with entries:
	$$M_2(x,m)_{1,j} = [j] \times B((x), j-1).$$
	
	Define $M_3$ as the $1 \times m$ matrix with entries:
	$$M_3(x,y,m)_{1,j} = B((y), j)-B((x), j).$$
\end{dfn}
\begin{dfn}
	Let $\alpha=(\alpha_1,\alpha_3,\ldots,\alpha_{2n-1})$ with $\alpha_i \ge 0$, $m = \sum \alpha_i$. Define the following $m \times m$ block matrices:
	\begin{align*}H_1((x_1,\ldots,x_{2n-1}), \alpha) := \begin{pmatrix}M_1(x_1,\alpha_1,m)\\M_2(x_2,m)\\M_1(x_3,\alpha_3,m)\\\vdots\\M_1(x_{2n-1},\alpha_{2n-1},m)\end{pmatrix}\text{ and }\\H_2((x_1,x_3,\ldots,x_{2n-1}), \alpha) := \begin{pmatrix}M_1(x_1,\alpha_1,m)\\M_3(x_1,x_3,m)\\M_1(x_3,\alpha_3,m)\\\vdots\\M_1(x_{2n-1},\alpha_{2n-1},m)\end{pmatrix}.\end{align*}
\end{dfn}

\begin{lemma}
	\label{lemmaPart1}
	Fix $X=(x_1,\ldots,x_n)$, $P=(p_1,\ldots,p_n)$ with $p_1<\ldots<p_{n}$ and $\alpha=(\alpha_1,\ldots,\alpha_{n})$.
	We have a sijection
	$$\{f \in \Topo(G_X(\alpha)) : f(u_i)=p_{i}\ \forall i\} \leftrightarrow \phi(D(H((),X,\alpha)),X,P)$$
	
\end{lemma}
\begin{proof}
	Follows directly from Lemmas \ref{lemmaTrick} and \ref{DoublyGeneralisedVandemonde}.
\end{proof}

\begin{lemma}
	\label{lemmaPart2}
	Fix $\alpha=(\alpha_1,\ldots,\alpha_{2n-1})$ with $\alpha_i \ge 0$ for $i$ odd and $\alpha_i=1$ for $i$ even. Also fix $X=(x_1,\ldots,x_{2n-1})$ a tuple of formal variables. Then we have a sijection
	\begin{align*}
	&[(\sum \alpha_i)!] \times D(H((),X,\alpha))\leftrightarrow \prod_{i\text{ odd}} [\alpha_i!] \times D(H_1(X,(\alpha_1,\alpha_3,\ldots,\alpha_{2n-1})))
	\end{align*}
\end{lemma}
\begin{proof}
	
	By Lemma \ref{lemmaBinomial1} we have:
	$$[j] \times M((),x,\alpha,m)_{i,j} = [j] \times B((\underbrace{x,\ldots,x}_{i\text{ times}}),j-i) \leftrightarrow [i] \times B((\underbrace{x,\ldots,x}_{i+1\text{ times}}),j-i) = [i] \times M_1(x,\alpha,m)$$
	We also have $[j] \times M((),x,1,m)_{1,j} = [j] \times B((x),j-1) = M_2(x,m)_{1,j}$
	
	From these two facts it follows that multiplying the $j$th column of $D(H((),\alpha,X))$ through by $[j]$ for all $j$ yields the same signed set as if one multiplies the $j$th row of the block $M_1(x_{2i-1},\alpha_{2i-1},m)$ through by $j$, for all $j$ and $i$.
\end{proof}
\begin{example}
	Let $\alpha=(1,1,2), X=(x_1,x_2,x_3)$.
	Then
	\begin{align*}
	&[4!]\times D(H((),X,\alpha))\\
	&\leftrightarrow [4!]\times D\left(\begin{pmatrix}
	B((x_1),0)&B((x_1),1)&B((x_1),2)&B((x_1),3)\\
	B((x_2),0)&B((x_2),1)&B((x_2),2)&B((x_2),3)\\
	B((x_3),0)&B((x_3),1)&B((x_3),2)&B((x_3),3)\\
	\emptyset&B((x_3,x_3),0)&B((x_3,x_3),1)&B((x_3,x_3),2)
	\end{pmatrix}\right)\\
	&\leftrightarrow D\left(\begin{pmatrix}
	[1]\times B((x_1),0)&[2]\times B((x_1),1)&[3]\times B((x_1),2)&[4]\times B((x_1),3)\\
	[1]\times B((x_2),0)&[2]\times B((x_2),1)&[3]\times B((x_2),2)&[4]\times B((x_2),3)\\
	[1]\times B((x_3),0)&[2]\times B((x_3),1)&[3]\times B((x_3),2)&[4]\times B((x_3),3)\\
	\emptyset&[2]\times B((x_3,x_3),0)&[3]\times B((x_3,x_3),1)&[4]\times B((x_3,x_3),2)
	\end{pmatrix}\right)\\
	&\leftrightarrow D\left(\begin{pmatrix}
	[1]\times B((x_1,x_1),0)&[1]\times B((x_1,x_1),1)&[1]\times B((x_1,x_1),2)&[1]\times B((x_1,x_1),3)\\
	[1]\times B((x_2),0)&[2]\times B((x_2),1)&[3]\times B((x_2),2)&[4]\times B((x_2),3)\\
	[1]\times B((x_3,x_3),0)&[1]\times B((x_3,x_3),1)&[1]\times B((x_3,x_3),2)&[1]\times B((x_3,x_3),3)\\
	\emptyset&[2]\times B((x_3,x_3,x_3),0)&[2]\times B((x_3,x_3,x_3),1)&[2]\times B((x_3,x_3,x_3),2)
	\end{pmatrix}\right)\\
	&\leftrightarrow [1!2!]\times D\left(\begin{pmatrix}
	B((x_1,x_1),0)&B((x_1,x_1),1)&B((x_1,x_1),2)&B((x_1,x_1),3)\\
	[1]\times B((x_2),0)&[2]\times B((x_2),1)&[3]\times B((x_2),2)&[4]\times B((x_2),3)\\
	B((x_3,x_3),0)&B((x_3,x_3),1)&B((x_3,x_3),2)&B((x_3,x_3),3)\\
	\emptyset&B((x_3,x_3,x_3),0)&B((x_3,x_3,x_3),1)&B((x_3,x_3,x_3),2)
	\end{pmatrix}\right)\\
	&\leftrightarrow [1!2!]\times D(H_1(X,(1,2)))
	\end{align*}
\end{example}

\begin{lemma}
	\label{lemmaPart3}
	Fix $\alpha=(\alpha_1,\alpha_3,\ldots,\alpha_{2n-1})$ with $\alpha_i \ge 0$. Also fix $X=(x_1,\ldots,x_{2n-1})$, $P=(p_1,p_3,\ldots,p_{2n-1})$ with $p_1<p_3<\ldots<p_{2n-1}$.
	\begin{align*}
	&\sum_{p_1<p_2<\ldots<p_{2n-1}} \phi(D(H_1(X,\alpha)),X,(p_1,p_2,\ldots,p_{2n-1}))\\
	&\leftrightarrow \phi(D(H_2((x_1,x_3,\ldots,x_{2n-1}),\alpha)),(x_1,x_3,\ldots,x_{2n-1}),P)
	\end{align*}
\end{lemma}
\begin{proof}
	By Lemma \ref{lemmaPhiIntegration} and Lemma \ref{lemmaPhiIntegration2} we have for fixed $j \in [m]$:
	$$\sum_{p_1<p_2<p_3} \phi(M_2(x_2,m)_{1,j},(x_1,x_2,x_3),(p_1,p_2,p_3)) \leftrightarrow \phi(M_3(x_1,x_3,m)_{1,j},(x_1,x_3),(p_1,p_3)).$$
	
	Therefore
	\begin{align*}
	&\sum_{p_1<p_2<\ldots<p_{2n-1}} \phi(D(H_1(X,\alpha)),X,(p_1,p_2,\ldots,p_{2n-1}))\\
	&\leftrightarrow \sum_{p_1<p_2<\ldots<p_{2n-1}} \phi(D(\begin{pmatrix}M_1(x_1,\alpha_1,m)\\M_2(x_2,m)\\M_1(x_3,\alpha_3,m)\\M_2(x_4,m)\\\vdots\\M_1(x_{2n-1},\alpha_{2n-1},m)\end{pmatrix}),X,(p_1,p_2,\ldots,p_{2n-1}))\\
	&\leftrightarrow \sum_{p_3<p_4<\ldots<p_{2n-1}} \phi(D(\begin{pmatrix}M_1(x_1,\alpha_1,m)\\M_3(x_1,x_3,m)\\M_1(x_3,\alpha_3,m)\\M_2(x_4,m)\\\vdots\\M_1(x_{2n-1},\alpha_{2n-1},m)\end{pmatrix}),(x_1,x_3,x_4,\ldots,x_{2n-1}),(p_1,p_3,p_4,\ldots,p_{2n-1}))
	\end{align*}
	where we condition on the exponent of $x_2$.
	We repeat this argument for every $p_i$ for $i$ even:	
	\begin{align*}
	&\leftrightarrow \phi(D(H_2((x_1,x_3,\ldots,x_{2n-1}),\alpha)),(x_1,x_3,\ldots,x_{2n-1}),P)
	\end{align*}
\end{proof}

\begin{lemma}
	\label{lemmaPart4}
	Fix $\alpha=(\alpha_1,\alpha_2,\ldots,\alpha_{n})$ with $\alpha_i \ge 0$. Also fix $X=(x_1,\ldots,x_n)$, $P=(p_1,\ldots,p_n)$ with $p_1<\ldots<p_{n}$.
	$$\phi(D(H((),X,(\alpha_1+1,\ldots,\alpha_n+1))),X,P)\leftrightarrow \phi(D(H_2(X,\alpha),X,P) $$
\end{lemma}
\begin{proof}
	We start with the matrix $H((),X,(\alpha_1+1,\ldots,\alpha_n+1))$ and do the following row operations:
	We subtract row $\alpha_1+\ldots+\alpha_{i-1}+1$ from row $\alpha_1+\ldots+\alpha_{i}+1$ for every $i=n-1,n-2,\ldots,2,1$ in this order.
	
	In the resulting matrix the entry in row $\alpha_1+\ldots+\alpha_{i}+1$ and column $j$ is
	$B((x_i),j-1)-B((x_{i-1}),j-1)$.
	For $j=1$, we have a sijection from $B((x_i),0)-B((x_{i-1}),0) \leftrightarrow \emptyset$.
	Therefore the first column contains only empty sets, except for the entry in row $1$, which contains a single positive element of weight $1$. Hence we can delete the first row and first column.
	
	Also note, that for $j\ge 2$, we have $B((x_i),j-1)-B((x_{i-1}),j-1) =  M_3(x_{i-1},x_i,m)_{1,j-1}$.
\end{proof}

\begin{proof}[Proof of Theorem \ref{crux}]
	In the definition of $A$, the values of $f$ for all odd $i$ are fixed. Therefore we can sum over all possible values of the even $i$:
	$$[(\sum \alpha_i)!] \times A \leftrightarrow \sum_{p_1<p_2<\ldots<p_{2n-1}} [(\sum \alpha_i)!] \times \{f \in \Topo(G_A) : f(u_i)=p_i\ \forall i\}$$
	where the sum is over all $p_i$ with even $i$. Now we apply Lemma \ref{lemmaPart1} to move from topological orders to determinants.
	$$\leftrightarrow \sum_{p_1<p_2<\ldots<p_{2n-1}} [(\sum \alpha_i)!] \times \phi(D(H((),(\alpha_1,\ldots,\alpha_{2n-1}),X)),X,(p_1,\ldots,p_{2n-1}).$$
	Now apply Lemma \ref{lemmaPart2}, which does some row operations on the matrix.
	$$\leftrightarrow \prod_{i\text{ odd}} [\alpha_i!]\times \sum_{p_1<p_2<\ldots<p_{2n-1}}  \phi(D(H_1(X,(\alpha_1,\alpha_3,\ldots,\alpha_{2n-1}))),X,(p_1,\ldots,p_{2n-1}).$$
	Apply Lemma \ref{lemmaPart3}, which corresponds to the integrals in Theorem \ref{cruxIntegral}.
	$$\leftrightarrow \prod_{i\text{ odd}} [\alpha_i!] \times \phi(D(H_2((x_1,x_3,\ldots,x_{2n-1}),(\alpha_1,\ldots,\alpha_{2n-1}))),(x_1,x_3,\ldots,x_{2n-1}),(p_1,p_3,\ldots,p_{2n-1}).$$
	Apply Lemma \ref{lemmaPart4}, which does some more row operations on the matrix.
	$$\leftrightarrow \prod_{i\text{ odd}} [\alpha_i!] \times \phi(D(H((),(x_1,x_3,\ldots,x_{2n-1}),(\alpha_1+1,\ldots,\alpha_{2n-1}+1))),(x_1,x_3,\ldots,x_{2n-1}),(p_1,p_3,\ldots,p_{2n-1}).$$
	Finally apply Lemma \ref{lemmaPart1} again, which relates the determinant back to topological orders:
	$$\leftrightarrow \prod_{i\text{ odd}} [\alpha_i!] \times B.$$

\end{proof}


\section{Proof of Theorem \ref{combSelberg}} \label{sectionSelberg}

\begin{dfn}
	\label{defGraphAPrime}
	For fixed $\alpha_1,\ldots,\alpha_{2n-1} \ge 0$ with $\alpha_i=1$ for even $i$, 
	let $G_A(\alpha_1,\alpha_3,\ldots,\alpha_{2n-1})=(V_A,E_A)$ be the graph with vertex set:
	\begin{align*}
	V_A = \{u_1,u_2,\ldots,u_{2n-1}\}&\\
	\cup\ \{w_{i,j,k} :\ &\forall i<j \in [2n-1],\text{ not both odd}\ \\
	&\forall k\in [\alpha_i\cdot \alpha_j] \}
	\end{align*}
	and edge set
	\begin{align*}E_A = \{u_{i+1} \rightarrow u_i : \forall i \in [2n-2]\}& \\
	\cup\ \{u_j \rightarrow w_{i,j,k} \rightarrow u_i :\ &\forall i<j \in [2n-1],\text{ not both odd}\\
	&\forall k\in [\alpha_i \cdot \alpha_j] \}.\end{align*}
\end{dfn}
\begin{dfn}
	\label{defGraphBPrime}
	For fixed $\alpha_1,\ldots,\alpha_{2n-1} \ge 0$ with $\alpha_i=1$ for even $i$, 
	let $G_B(\alpha_1,\alpha_3,\ldots,\alpha_{2n-1})=(V_B,E_B)$ be the graph with vertex set:
	\begin{align*}V_B = \{u_1,u_3,\ldots,u_{2n-1}\}&\\
	\cup\ \{w_{i,j,k} :\ &\forall i<j \in \{1,3,\ldots,2n-1\}\\
	&\forall k\in [\alpha_i + \alpha_j + 1]\}\end{align*} and edge set
	\begin{align*}E_B = \{u_{i+2} \rightarrow u_i : \forall i \in [2n-3]\}& \\
	\cup\ \{u_j \rightarrow w_{i,j,k} \rightarrow u_i :\ &\forall i<j \in \{1,3,\ldots,2n-1\}\\
	&\forall k\in [\alpha_i + \alpha_j + 1] \}.\end{align*}
\end{dfn}
\begin{cor}
	\label{cruxPrime}
	Fix $\alpha=(\alpha_1,\ldots,\alpha_{2n-1})$ with $\alpha_i \ge 0$ for $i$ odd and $\alpha_i=1$ for even $i$ and let $G_A$ and $G_B$ be as in Definitions \ref{defGraphAPrime} and \ref{defGraphBPrime}. Fix any $p_1,p_3,\ldots,p_{2n-1}$. Then we have
	$$(\sum \alpha_i)! \cdot |\{ f \in \Topo(G_A)\colon f(u_i)=p_i, \forall i\ \text{odd}\}| = \prod \alpha_i! \cdot |\{ f \in \Topo(G_B)\colon f(u_i)=p_i, \forall i\ \text{odd}\}|.$$
\end{cor}
\begin{proof}
	One way of proving this is by momentarily consider polynomials: Theorem \ref{crux} is equivalent to Theorem \ref{cruxIntegral} which is equivalent to Corollary \ref{cruxIntegralCorollary} which again is equivalent to Corollary \ref{cruxPrime}.
	
	We can prove it in an alternate way: If this is not true for all tuples $(p_1,p_3,\ldots,p_{2n-1})$. Choose minimal counterexample, with respect to the lexicographical ordering on the tuples $(p_1,p_3,\ldots,p_{2n-1})$.
	
	Define $p_1',p_3',\ldots,p_{2n-1}'$ as follows:
	$$p_k' := p_k + \sum_{\substack{1\le i <j \le k\\i,j\text{ odd}}} \alpha_i \alpha_j$$
	and consider
	$$\{f \in \Topo(G_X(\alpha))\colon f(u_i)=p_i',\forall i\text{ odd}\}.$$
	We can enumerate this set, by first considering all possible labels of the vertices $w_{i,j,k}$ for $i$ and $j$ odd, and then the rest, which is a topological order of $G_A$, for which the labels of $u_1,u_3,\ldots,u_{2n-1}$ has been fixed to be $(p_1'',p_3'',\ldots,p_{2n-1}'')$. We can do the same for $G_X(\alpha_1+1,\alpha_3+1,\ldots,\alpha_{2n-1}+1)$ with $G_B$. The cruxial point here is that although $p''$ depends on the choices of the $f(w_{i,j,k})$, we always have $(p_1'',p_3'',\ldots,p_{2n-1}'') \le (p_1,\ldots,p_{2n-1})$. Therefore by Theorem \ref{crux} and minimality of $(p_1,p_3,\ldots,p_{2n-1})$, we have that Corollary \ref{cruxPrime} holds for $(p_1,p_3,\ldots,p_{2n-1})$ also, a contradiction.
\end{proof}
\begin{rem}
	One could have stated Corollary \ref{cruxPrime} as a bijection. But above proof would result in a bijection $$LHS \times X \leftrightarrow RHS \times X,$$
	for some large set $X$, which we cannot simply divide through.
\end{rem}

Finally, we define the last graph $G_K$ which we relate to $G_S$ in two different ways. This part of the proof is analogous to Anderson's Proof \cite{Anderson1991}, except that rather than defining a graph $G_K$, they instead define an integral $K$.
\begin{dfn}    
	\label{defKGraph}
	For fixed integers $n,a,b,c \ge 1$ let $G_K(n,a,b,c)=(V_K,E_K)$ be the graph with vertex set:
	\begin{align*}
	V_K = &\ \{x_1,x_2,\ldots,x_n\}\\
	&\cup\ \{y_1,y_2,\ldots,y_{n-1}\}\\
	&\cup\ \{v_{i,j} :\ \forall i<j \in [n]\}\\
	&\cup\ \{w_{i,j} :\ \forall i<j \in [n-1]\}\\
	&\cup\ \{p_{i,k} :\ \forall i \in [n]\ \forall k \in [a-1]\}\\
	&\cup\ \{q_{i,k} :\ \forall i \in [n]\ \forall k \in [b-1]\}\\
	&\cup\ \{r_{i,j,k} :\ \forall i \in [n]\ \forall j \in [n-1]\ \forall k \in [c-1]\}
	\end{align*}
	and edge set
	\begin{align*}E_K = &\ \{x_{i+1} \rightarrow y_i \rightarrow x_i : \forall i \in [n-1]\} \\
	&\cup\ \{x_j \rightarrow v_{i,j} \rightarrow x_i :\ \forall i<j \in [n]\}\\
	&\cup\ \{y_j \rightarrow w_{i,j} \rightarrow y_i :\ \forall i<j \in [n-1]\}\\
	&\cup\ \{x_i \rightarrow p_{i,k} :\ \forall i \in [n]\ \forall k \in [a-1]\}\\
	&\cup\ \{q_{i,k} \rightarrow x_i :\ \forall i \in [n]\ \forall k \in [b-1]\}\\
	&\cup\ \{y_j \rightarrow r_{i,j,k} \rightarrow x_i :\ \forall j \in [n-1]\ \forall i \in [j]\ \forall k\in [c-1] \}\\
	&\cup\ \{y_j \leftarrow r_{i,j,k} \leftarrow x_i :\ \forall i \in [n]\ \forall j \in [i-1]\ \forall k\in [c-1] \}.\end{align*}  
\end{dfn}

\begin{lemma}
	For fixed integers $n,a,b,c \ge 1$ let $G_K=G_K(n,a,b,c)$ and $G_S=G_S(n,a,b,c)$. Then we have
	$$\topo(G_K) = \frac{(c-1)!^n}{(nc-1)!} \cdot \topo(G_S).$$
\end{lemma}
\begin{proof}
	Note that the induced subgraph $H$ of $K_G$ on the vertices $x_i, y_i, w_{i,j}, r_{i,j,k}$ is isomorphic to the graph $G_A(c-1,\ldots,c-1)$ (c.f. Definition \ref{defGraphAPrime}). Replacing this copy of $G_A$ within $G_K$ by the graph $G_B$ (keeping the old vertices $x_1,\ldots,x_n$) yields the graph $G_S$ (c.f. Definition \ref{defSGraph}). The result follows immediately by Theorem \ref{cruxPrime}.
\end{proof}
\begin{lemma}
	For fixed integers $n,a,b,c \ge 1$ let $G_K=G_K(n,a,b,c)$ and $G_S=G_S(n-1,a+c,b+c,c)$. Then we have
	$$\topo(G_K)\\
	= \frac{(n (a+b+c (n-1)-1))!}{((n-1) (a + b + c n-1))!} \frac{(a-1)! (b-1)! (c-1)!^{n-1}}{(a+b+(n-1)c-1)!} \cdot \topo(G_S).$$
\end{lemma}
\begin{proof}
	First construct from $G_K$ the graph $G_K'$ by adding vertices $y_0$ and $y_n$ to the vertex set and edges $x_1 \rightarrow y_0, y_n \rightarrow x_n$ and $p_{i,k} \rightarrow y_0, y_n \rightarrow q_{i,k}\ \forall i,k$ to the edge set. It is easy to see that $\topo(G_K) = \topo(G_K')$, because vertices $y_0$ (resp. $y_n$) need to be assigned the smallest (resp. largest) label.
	Then note that the induced subgraph $H$ of $K_G'$ on the vertices $x_i, y_i, v_{i,j}, p_{i,k}, q_{i,k}, r_{i,j,k}$ is isomorphic to the graph $G_A(a-1,c-1,\ldots,c-1,b-1)$ (c.f. Definition \ref{defGraphAPrime}). Replacing this copy of $G_A$ within $G_K'$ by the graph $G_B$ (keeping the old vertices $y_0,\ldots,y_n$) and removing $y_0$ and $y_n$ yields the graph $G_S$ (c.f. Definition \ref{defSGraph}) together with $a+b-1$ isolated vertices. 
	Now, there are $\frac{(n (a+b+c (n-1)-1))!}{((n-1) (a + b + c n-1))!}$ ways of choosing labels for these $a+b-1$ vertices. The other factor follows immediately by Theorem \ref{cruxPrime}.
\end{proof}

Now Theorem \ref{combSelberg} follows by chaining these Lemmas together and using induction  on $n$ with base case $\topo(G_S(0,a,b,c))=1$:
\begin{align*}
&\topo(G_S(n,a,b,c))\\
&= \frac{(n (a+b+c (n-1)-1))!}{((n-1) (a + b + c n-1))!} \frac{(nc-1)! (a-1)! (b-1)!}{(c-1)! (a+b+(n-1)c-1)!} \\
&\quad \quad \quad \quad \quad \times \topo(G_S(n-1,a+c,b+c,c)).
\end{align*}

We can illustrate this using an example:
\begin{example}
	Let $n=3, a=3, b= 2,c=1$, 
	We want to evaluate $\topo(G_S(3,3,2,1))$ and start by conditioning on the labels of the red vertices as follows:
	
	$$\topo\left(
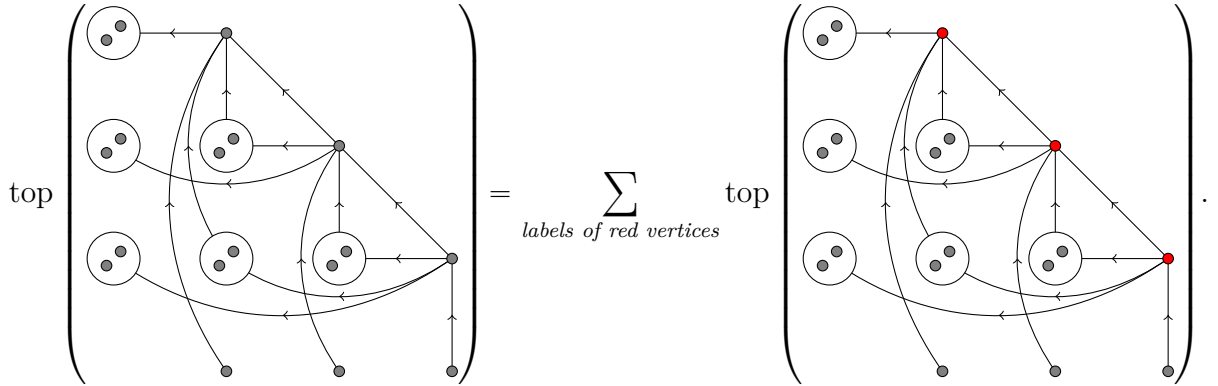
\tikz[baseline=(p)]{
		\tikzset{scale=0.750000,yscale=-1,shorten >=2pt,decoration={markings,mark=at position 0.5 with {\arrow{>}}}}
		\tikzset{every node/.style={circle,draw,solid,fill=black!50,inner sep=0pt,minimum width=4pt}}
		\draw[postaction={decorate}] (2,2) to (0,0);
		\draw[postaction={decorate}] (4,4) to (2,2);
		
		\draw[postaction={decorate}] (2,2) to (0,2);
		\draw[postaction={decorate}] (0,2) to (0,0);
		\draw[postaction={decorate},bend right=35] (4,4) to (0,4);
		\draw[postaction={decorate},bend right=35] (0,4) to (0,0);
		\draw[postaction={decorate}] (4,4) to (2,4);
		\draw[postaction={decorate}] (2,4) to (2,2);
		
		\draw[postaction={decorate},bend right=35] (0,6) to (0,0);
		\draw[postaction={decorate},bend right=35] (2,6) to (2,2);
		\draw[postaction={decorate}] (4,6) to (4,4);
		\draw[postaction={decorate}] (0,0) to (-2,0);
		\draw[postaction={decorate},bend right=35] (2,2) to (-2,2);
		\draw[postaction={decorate},bend right=35] (4,4) to (-2,4);
		
		\draw (-2,0) node[draw=black,fill=white,minimum width=20pt] {};
		\draw (-2,2) node[draw=black,fill=white,minimum width=20pt] {};
		\draw (-2,4) node[draw=black,fill=white,minimum width=20pt] {};
		\draw (0,2) node[draw=black,fill=white,minimum width=20pt] {};
		\draw (0,4) node[draw=black,fill=white,minimum width=20pt] {};
		\draw (2,4) node[draw=black,fill=white,minimum width=20pt] {};
		
		\draw (0,0) node[draw=black,fill=black!50] {};
		\draw (2,2) node[draw=black,fill=black!50] {};
		\draw (4,4) node[draw=black,fill=black!50] {};
		\draw (-0.125,2.125) node[draw=black,fill=black!50] {};
		\draw (0.125,1.875) node[draw=black,fill=black!50] {};
		\draw (-0.125,4.125) node[draw=black,fill=black!50] {};
		\draw (0.125,3.875) node[draw=black,fill=black!50] {};
		\draw (1.875,4.125) node[draw=black,fill=black!50] {};
		\draw (2.125,3.875) node[draw=black,fill=black!50] {};
		\draw (0,6) node[draw=black,fill=black!50] {};
		\draw (2,6) node[draw=black,fill=black!50] {};
		\draw (4,6) node[draw=black,fill=black!50] {};
		\draw (-2.125,0.125) node[draw=black,fill=black!50] {};
		\draw (-1.875,-0.125) node[draw=black,fill=black!50] {};
		\draw (-2.125,2.125) node[draw=black,fill=black!50] {};
		\draw (-1.875,1.875) node[draw=black,fill=black!50] {};
		\draw (-2.125,4.125) node[draw=black,fill=black!50] {};
		\draw (-1.875,3.875) node[draw=black,fill=black!50] {};
		\coordinate (p) at ([yshift=1ex]current bounding box.center);
	}
	\right)
	=\sum_\text{labels of red vertices} \topo\left(\tikz[baseline=(p)]{
		\tikzset{scale=0.750000,yscale=-1,shorten >=2pt,decoration={markings,mark=at position 0.5 with {\arrow{>}}}}
		\tikzset{every node/.style={circle,draw,solid,fill=black!50,inner sep=0pt,minimum width=4pt}}
		\draw[postaction={decorate}] (2,2) to (0,0);
		\draw[postaction={decorate}] (4,4) to (2,2);
		
		\draw[postaction={decorate}] (2,2) to (0,2);
		\draw[postaction={decorate}] (0,2) to (0,0);
		\draw[postaction={decorate},bend right=35] (4,4) to (0,4);
		\draw[postaction={decorate},bend right=35] (0,4) to (0,0);
		\draw[postaction={decorate}] (4,4) to (2,4);
		\draw[postaction={decorate}] (2,4) to (2,2);
		
		\draw[postaction={decorate},bend right=35] (0,6) to (0,0);
		\draw[postaction={decorate},bend right=35] (2,6) to (2,2);
		\draw[postaction={decorate}] (4,6) to (4,4);
		\draw[postaction={decorate}] (0,0) to (-2,0);
		\draw[postaction={decorate},bend right=35] (2,2) to (-2,2);
		\draw[postaction={decorate},bend right=35] (4,4) to (-2,4);
		
		\draw (-2,0) node[draw=black,fill=white,minimum width=20pt] {};
		\draw (-2,2) node[draw=black,fill=white,minimum width=20pt] {};
		\draw (-2,4) node[draw=black,fill=white,minimum width=20pt] {};
		\draw (0,2) node[draw=black,fill=white,minimum width=20pt] {};
		\draw (0,4) node[draw=black,fill=white,minimum width=20pt] {};
		\draw (2,4) node[draw=black,fill=white,minimum width=20pt] {};
		
		\draw (0,0) node[draw=black,fill=red] {};
		\draw (2,2) node[draw=black,fill=red] {};
		\draw (4,4) node[draw=black,fill=red] {};
		\draw (-0.125,2.125) node[draw=black,fill=black!50] {};
		\draw (0.125,1.875) node[draw=black,fill=black!50] {};
		\draw (-0.125,4.125) node[draw=black,fill=black!50] {};
		\draw (0.125,3.875) node[draw=black,fill=black!50] {};
		\draw (1.875,4.125) node[draw=black,fill=black!50] {};
		\draw (2.125,3.875) node[draw=black,fill=black!50] {};
		\draw (0,6) node[draw=black,fill=black!50] {};
		\draw (2,6) node[draw=black,fill=black!50] {};
		\draw (4,6) node[draw=black,fill=black!50] {};
		\draw (-2.125,0.125) node[draw=black,fill=black!50] {};
		\draw (-1.875,-0.125) node[draw=black,fill=black!50] {};
		\draw (-2.125,2.125) node[draw=black,fill=black!50] {};
		\draw (-1.875,1.875) node[draw=black,fill=black!50] {};
		\draw (-2.125,4.125) node[draw=black,fill=black!50] {};
		\draw (-1.875,3.875) node[draw=black,fill=black!50] {};
		\coordinate (p) at ([yshift=1ex]current bounding box.center);
	}
	\right).$$
	
	Having fixed the labels of the red vertices, we can apply Theorem \ref{cruxPrime} to each summand.
	
	$$=\sum_\text{labels of red vertices} 2\cdot \topo\left(\tikz[baseline=(p)]{
		\tikzset{scale=0.750000,yscale=-1,shorten >=2pt,decoration={markings,mark=at position 0.5 with {\arrow{>}}}}
		\tikzset{every node/.style={circle,draw,solid,fill=black!50,inner sep=0pt,minimum width=4pt}}
		\draw[postaction={decorate}] (1,1) to (0,0);
		\draw[postaction={decorate}] (2,2) to (1,1);
		\draw[postaction={decorate}] (3,3) to (2,2);
		\draw[postaction={decorate}] (4,4) to (3,3);
		
		\draw[postaction={decorate}] (2,2) to (0,2);
		\draw[postaction={decorate}] (0,2) to (0,0);
		\draw[postaction={decorate},bend right=20] (4,4) to (0,4);
		\draw[postaction={decorate},bend right=20] (0,4) to (0,0);
		\draw[postaction={decorate}] (4,4) to (2,4);
		\draw[postaction={decorate}] (2,4) to (2,2);
		
		\draw[postaction={decorate},bend right=25] (0,6) to (0,0);
		\draw[postaction={decorate},bend right=20] (2,6) to (2,2);
		\draw[postaction={decorate}] (4,6) to (4,4);
		\draw[postaction={decorate}] (0,0) to (-2,0);
		\draw[postaction={decorate},bend right=20] (2,2) to (-2,2);
		\draw[postaction={decorate},bend right=25] (4,4) to (-2,4);
		\draw[postaction={decorate},bend right=10] (3,3) to (1,3);
		\draw[postaction={decorate},bend right=10] (1,3) to (1,1);
		
		\draw (-2,0) node[draw=black,fill=white,minimum width=20pt] {};
		\draw (-2,2) node[draw=black,fill=white,minimum width=20pt] {};
		\draw (-2,4) node[draw=black,fill=white,minimum width=20pt] {};
		\draw (0,2) node[draw=black,fill=black!50] {};
		\draw (0,4) node[draw=black,fill=black!50] {};
		\draw (2,4) node[draw=black,fill=black!50] {};
		
		\draw (0,0) node[draw=black,fill=red] {};
		\draw (1,1) node[draw=black,fill=black!50] {};
		\draw (2,2) node[draw=black,fill=red] {};
		\draw (3,3) node[draw=black,fill=black!50] {};
		\draw (4,4) node[draw=black,fill=red] {};
		\draw (0,6) node[draw=black,fill=black!50] {};
		\draw (2,6) node[draw=black,fill=black!50] {};
		\draw (4,6) node[draw=black,fill=black!50] {};
		\draw (-2.125,0.125) node[draw=black,fill=black!50] {};
		\draw (-1.875,-0.125) node[draw=black,fill=black!50] {};
		\draw (-2.125,2.125) node[draw=black,fill=black!50] {};
		\draw (-1.875,1.875) node[draw=black,fill=black!50] {};
		\draw (-2.125,4.125) node[draw=black,fill=black!50] {};
		\draw (-1.875,3.875) node[draw=black,fill=black!50] {};
		\draw (1,3) node[draw=black,fill=black!50] {};
		\coordinate (p) at ([yshift=1ex]current bounding box.center);
	}
	\right).$$
	
	Undoing the first step then yields:
	$$
	=2\cdot \topo\left(\tikz[baseline=(p)]{
		\tikzset{scale=0.750000,yscale=-1,shorten >=2pt,decoration={markings,mark=at position 0.5 with {\arrow{>}}}}
		\tikzset{every node/.style={circle,draw,solid,fill=black!50,inner sep=0pt,minimum width=4pt}}
		\draw[postaction={decorate}] (1,1) to (0,0);
		\draw[postaction={decorate}] (2,2) to (1,1);
		\draw[postaction={decorate}] (3,3) to (2,2);
		\draw[postaction={decorate}] (4,4) to (3,3);
		
		\draw[postaction={decorate}] (2,2) to (0,2);
		\draw[postaction={decorate}] (0,2) to (0,0);
		\draw[postaction={decorate},bend right=20] (4,4) to (0,4);
		\draw[postaction={decorate},bend right=20] (0,4) to (0,0);
		\draw[postaction={decorate}] (4,4) to (2,4);
		\draw[postaction={decorate}] (2,4) to (2,2);
		
		\draw[postaction={decorate},bend right=25] (0,6) to (0,0);
		\draw[postaction={decorate},bend right=20] (2,6) to (2,2);
		\draw[postaction={decorate}] (4,6) to (4,4);
		\draw[postaction={decorate}] (0,0) to (-2,0);
		\draw[postaction={decorate},bend right=20] (2,2) to (-2,2);
		\draw[postaction={decorate},bend right=25] (4,4) to (-2,4);
		\draw[postaction={decorate},bend right=10] (3,3) to (1,3);
		\draw[postaction={decorate},bend right=10] (1,3) to (1,1);
		
		\draw (-2,0) node[draw=black,fill=white,minimum width=20pt] {};
		\draw (-2,2) node[draw=black,fill=white,minimum width=20pt] {};
		\draw (-2,4) node[draw=black,fill=white,minimum width=20pt] {};
		\draw (0,2) node[draw=black,fill=black!50] {};
		\draw (0,4) node[draw=black,fill=black!50] {};
		\draw (2,4) node[draw=black,fill=black!50] {};
		
		\draw (0,0) node[draw=black,fill=black!50] {};
		\draw (1,1) node[draw=black,fill=black!50] {};
		\draw (2,2) node[draw=black,fill=black!50] {};
		\draw (3,3) node[draw=black,fill=black!50] {};
		\draw (4,4) node[draw=black,fill=black!50] {};
		\draw (0,6) node[draw=black,fill=black!50] {};
		\draw (2,6) node[draw=black,fill=black!50] {};
		\draw (4,6) node[draw=black,fill=black!50] {};
		\draw (-2.125,0.125) node[draw=black,fill=black!50] {};
		\draw (-1.875,-0.125) node[draw=black,fill=black!50] {};
		\draw (-2.125,2.125) node[draw=black,fill=black!50] {};
		\draw (-1.875,1.875) node[draw=black,fill=black!50] {};
		\draw (-2.125,4.125) node[draw=black,fill=black!50] {};
		\draw (-1.875,3.875) node[draw=black,fill=black!50] {};
		\draw (1,3) node[draw=black,fill=black!50] {};
		\coordinate (p) at ([yshift=1ex]current bounding box.center);
	}
	\right).$$
	
	We introduce two new vertices which have to be labelled with the minimum and maximum label and again, condition on the red (now different) vertices, as follows:	
	$$
	=2\cdot \sum_\text{labels of red vertices} \topo\left(\tikz[baseline=(p)]{
		\tikzset{scale=0.750000,yscale=-1,shorten >=2pt,decoration={markings,mark=at position 0.5 with {\arrow{>}}}}
		\tikzset{every node/.style={circle,draw,solid,fill=black!50,inner sep=0pt,minimum width=4pt}}
		\draw[postaction={decorate}] (0,0) to (-1,-1);
		\draw[postaction={decorate}] (1,1) to (0,0);
		\draw[postaction={decorate}] (2,2) to (1,1);
		\draw[postaction={decorate}] (3,3) to (2,2);
		\draw[postaction={decorate}] (4,4) to (3,3);
		\draw[postaction={decorate}] (5,5) to (4,4);
		
		\draw[postaction={decorate}] (2,2) to (0,2);
		\draw[postaction={decorate}] (0,2) to (0,0);
		\draw[postaction={decorate},bend right=10] (4,4) to (0,4);
		\draw[postaction={decorate},bend right=10] (0,4) to (0,0);
		\draw[postaction={decorate}] (4,4) to (2,4);
		\draw[postaction={decorate}] (2,4) to (2,2);
		
		\draw[postaction={decorate},bend right=15] (5,5) to (0,5);
		\draw[postaction={decorate},bend right=15] (0,5) to (0,0);
		\draw[postaction={decorate},bend right=10] (5,5) to (2,5);
		\draw[postaction={decorate},bend right=10] (2,5) to (2,2);
		\draw[postaction={decorate}] (5,5) to (4,5);
		\draw[postaction={decorate}] (4,5) to (4,4);
		\draw[postaction={decorate}] (0,0) to (-1,0);
		\draw[postaction={decorate}] (-1,0) to (-1,-1);
		\draw[postaction={decorate},bend right=20] (2,2) to (-1,2);
		\draw[postaction={decorate},bend right=40] (-1,2) to (-1,-1);
		\draw[postaction={decorate},bend right=25] (4,4) to (-1,4);
		\draw[postaction={decorate},bend right=45] (-1,4) to (-1,-1);
		\draw[postaction={decorate},bend right=10] (3,3) to (1,3);
		\draw[postaction={decorate},bend right=10] (1,3) to (1,1);
		
		\draw (-1,0) node[draw=black,fill=white,minimum width=20pt] {};
		\draw (-1,2) node[draw=black,fill=white,minimum width=20pt] {};
		\draw (-1,4) node[draw=black,fill=white,minimum width=20pt] {};
		\draw (0,2) node[draw=black,fill=black!50] {};
		\draw (0,4) node[draw=black,fill=black!50] {};
		\draw (2,4) node[draw=black,fill=black!50] {};
		
		\draw (-1,-1) node[draw=black,fill=red] {};
		\draw (0,0) node[draw=black,fill=black!50] {};
		\draw (1,1) node[draw=black,fill=red] {};
		\draw (2,2) node[draw=black,fill=black!50] {};
		\draw (3,3) node[draw=black,fill=red] {};
		\draw (4,4) node[draw=black,fill=black!50] {};
		\draw (5,5) node[draw=black,fill=red] {};
		\draw (0,5) node[draw=black,fill=black!50] {};
		\draw (2,5) node[draw=black,fill=black!50] {};
		\draw (4,5) node[draw=black,fill=black!50] {};
		\draw (-1.125,0.125) node[draw=black,fill=black!50] {};
		\draw (-0.875,-0.125) node[draw=black,fill=black!50] {};
		\draw (-1.125,2.125) node[draw=black,fill=black!50] {};
		\draw (-0.875,1.875) node[draw=black,fill=black!50] {};
		\draw (-1.125,4.125) node[draw=black,fill=black!50] {};
		\draw (-0.875,3.875) node[draw=black,fill=black!50] {};
		\draw (1,3) node[draw=black,fill=black!50] {};
		\coordinate (p) at ([yshift=1ex]current bounding box.center);
	}
	\right).$$
	Applying Theorem \ref{cruxPrime} yields:
	$$=\frac{1}{180}\cdot \sum_\text{labels of red vertices} \topo\left(\tikz[baseline=(p)]{
		\tikzset{scale=0.750000,yscale=-1,shorten >=2pt,decoration={markings,mark=at position 0.5 with {\arrow{>}}}}
		\tikzset{every node/.style={circle,draw,solid,fill=black!50,inner sep=0pt,minimum width=4pt}}
		\draw[postaction={decorate}] (1,1) to (-1,-1);
		\draw[postaction={decorate}] (3,3) to (1,1);
		\draw[postaction={decorate}] (5,5) to (3,3);
		
		\draw[postaction={decorate}] (1,1) to (-1,1);
		\draw[postaction={decorate}] (-1,1) to (-1,-1);
		\draw[postaction={decorate}] (3,3) to (1,3);
		\draw[postaction={decorate}] (1,3) to (1,1);
		\draw[postaction={decorate}] (5,5) to (3,5);
		\draw[postaction={decorate}] (3,5) to (3,3);
		\draw[postaction={decorate},bend right=35] (3,3) to (-1,3);
		\draw[postaction={decorate},bend right=35] (-1,3) to (-1,-1);
		\draw[postaction={decorate},bend right=35] (5,5) to (1,5);
		\draw[postaction={decorate},bend right=35] (1,5) to (1,1);
		\draw[postaction={decorate},bend right=35] (5,5) to (-1,5);
		\draw[postaction={decorate},bend right=35] (-1,5) to (-1,-1);
		
		\draw (-1,1) node[draw=black,fill=white,minimum width=25pt] {};
		\draw (1,3) node[draw=black,fill=white,minimum width=20pt] {};
		\draw (3,5) node[draw=black,fill=white,minimum width=20pt] {};
		\draw (-1,3) node[draw=black,fill=white,minimum width=25pt] {};
		\draw (1,5) node[draw=black,fill=white,minimum width=20pt] {};
		\draw (-1,5) node[draw=black,fill=white,minimum width=30pt] {};
		
		\draw (-1,-1) node[draw=black,fill=red] {};
		\draw (1,1) node[draw=black,fill=red] {};
		\draw (3,3) node[draw=black,fill=red] {};
		\draw (5,5) node[draw=black,fill=red] {};
		
		\draw (-1.25,1.25) node[draw=black,fill=black!50] {};
		\draw (-1,1) node[draw=black,fill=black!50] {};
		\draw (-0.75,0.75) node[draw=black,fill=black!50] {};
		
		\draw (0.875,3.125) node[draw=black,fill=black!50] {};
		\draw (1.125,2.875) node[draw=black,fill=black!50] {};
		
		\draw (2.875,5.125) node[draw=black,fill=black!50] {};
		\draw (3.125,4.875) node[draw=black,fill=black!50] {};
		
		\draw (-1.25,3.25) node[draw=black,fill=black!50] {};
		\draw (-1,3) node[draw=black,fill=black!50] {};
		\draw (-0.75,2.75) node[draw=black,fill=black!50] {};
		
		\draw (0.875,5.125) node[draw=black,fill=black!50] {};
		\draw (1.125,4.875) node[draw=black,fill=black!50] {};
		
		\draw (-1.375,5.375) node[draw=black,fill=black!50] {};
		\draw (-1.125,5.125) node[draw=black,fill=black!50] {};
		\draw (-0.875,4.875) node[draw=black,fill=black!50] {};
		\draw (-0.625,4.625) node[draw=black,fill=black!50] {};
		\coordinate (p) at ([yshift=1ex]current bounding box.center);
	}
	\right).$$
	
	Next we remove again the new vertices: 
	$$=\frac{1}{180}\cdot \topo\left(\tikz[baseline=(p)]{
		\tikzset{scale=0.750000,yscale=-1,shorten >=2pt,decoration={markings,mark=at position 0.5 with {\arrow{>}}}}
		\tikzset{every node/.style={circle,draw,solid,fill=black!50,inner sep=0pt,minimum width=4pt}}
		\draw[postaction={decorate}] (3,3) to (1,1);
		
		\draw[postaction={decorate}] (1,1) to (-1,1);
		\draw[postaction={decorate}] (3,3) to (1,3);
		\draw[postaction={decorate}] (1,3) to (1,1);
		\draw[postaction={decorate}] (3,5) to (3,3);
		\draw[postaction={decorate},bend right=35] (3,3) to (-1,3);
		\draw[postaction={decorate},bend right=35] (1,5) to (1,1);
		
		\draw (-1,1) node[draw=black,fill=white,minimum width=25pt] {};
		\draw (1,3) node[draw=black,fill=white,minimum width=20pt] {};
		\draw (3,5) node[draw=black,fill=white,minimum width=20pt] {};
		\draw (-1,3) node[draw=black,fill=white,minimum width=25pt] {};
		\draw (1,5) node[draw=black,fill=white,minimum width=20pt] {};
		\draw (-1,5) node[draw=black,fill=white,minimum width=30pt] {};
		
		\draw (1,1) node[draw=black,fill=black!50] {};
		\draw (3,3) node[draw=black,fill=black!50] {};
		
		\draw (-1.25,1.25) node[draw=black,fill=black!50] {};
		\draw (-1,1) node[draw=black,fill=black!50] {};
		\draw (-0.75,0.75) node[draw=black,fill=black!50] {};
		
		\draw (0.875,3.125) node[draw=black,fill=black!50] {};
		\draw (1.125,2.875) node[draw=black,fill=black!50] {};
		
		\draw (2.875,5.125) node[draw=black,fill=black!50] {};
		\draw (3.125,4.875) node[draw=black,fill=black!50] {};
		
		\draw (-1.25,3.25) node[draw=black,fill=black!50] {};
		\draw (-1,3) node[draw=black,fill=black!50] {};
		\draw (-0.75,2.75) node[draw=black,fill=black!50] {};
		
		\draw (0.875,5.125) node[draw=black,fill=black!50] {};
		\draw (1.125,4.875) node[draw=black,fill=black!50] {};
		
		\draw (-1.375,5.375) node[draw=black,fill=black!50] {};
		\draw (-1.125,5.125) node[draw=black,fill=black!50] {};
		\draw (-0.875,4.875) node[draw=black,fill=black!50] {};
		\draw (-0.625,4.625) node[draw=black,fill=black!50] {};
		\coordinate (p) at ([yshift=1ex]current bounding box.center);
	}
	\right)$$
	
	There are $18\cdot 17\cdot 16\cdot 15$ possibilities for the four independent vertices, and hence we get:
	$$=408\cdot \topo\left(\tikz[baseline=(p)]{
		\tikzset{scale=0.750000,yscale=-1,shorten >=2pt,decoration={markings,mark=at position 0.5 with {\arrow{>}}}}
		\tikzset{every node/.style={circle,draw,solid,fill=black!50,inner sep=0pt,minimum width=4pt}}
		\draw[postaction={decorate}] (3,3) to (1,1);
		
		\draw[postaction={decorate}] (1,1) to (-1,1);
		\draw[postaction={decorate}] (3,3) to (1,3);
		\draw[postaction={decorate}] (1,3) to (1,1);
		\draw[postaction={decorate}] (3,5) to (3,3);
		\draw[postaction={decorate},bend right=35] (3,3) to (-1,3);
		\draw[postaction={decorate},bend right=35] (1,5) to (1,1);
		
		\draw (-1,1) node[draw=black,fill=white,minimum width=25pt] {};
		\draw (1,3) node[draw=black,fill=white,minimum width=20pt] {};
		\draw (3,5) node[draw=black,fill=white,minimum width=20pt] {};
		\draw (-1,3) node[draw=black,fill=white,minimum width=25pt] {};
		\draw (1,5) node[draw=black,fill=white,minimum width=20pt] {};
		
		\draw (1,1) node[draw=black,fill=black!50] {};
		\draw (3,3) node[draw=black,fill=black!50] {};
		
		\draw (-1.25,1.25) node[draw=black,fill=black!50] {};
		\draw (-1,1) node[draw=black,fill=black!50] {};
		\draw (-0.75,0.75) node[draw=black,fill=black!50] {};
		
		\draw (0.875,3.125) node[draw=black,fill=black!50] {};
		\draw (1.125,2.875) node[draw=black,fill=black!50] {};
		
		\draw (2.875,5.125) node[draw=black,fill=black!50] {};
		\draw (3.125,4.875) node[draw=black,fill=black!50] {};
		
		\draw (-1.25,3.25) node[draw=black,fill=black!50] {};
		\draw (-1,3) node[draw=black,fill=black!50] {};
		\draw (-0.75,2.75) node[draw=black,fill=black!50] {};
		
		\draw (0.875,5.125) node[draw=black,fill=black!50] {};
		\draw (1.125,4.875) node[draw=black,fill=black!50] {};
		
		\coordinate (p) at ([yshift=1ex]current bounding box.center);
	}
	\right)=408\cdot \topo(G_S(2,4,3,1))
	$$
	
\end{example}

\section{Conclusion}
We have presented a bijective proof of Theorem \ref{crux} or equivalently Corollary \ref{cruxIntegral}. Unfortunately, because of the \enquote{division} going to Corollary \ref{cruxPrime}, we don't have a bijection from topological orders of $G_S(n,a,b,c)$, but only a bijection from $G_S(n,a,b,c) \times X$ for some set $X$. A direct bijective proof for Theorem \ref{cruxPrime} would be very interesting.

\printbibliography

\addcontentsline{toc}{section}{Bibliography}

\newpage
\appendix
\section{\\Further Generalisation of the Vandermonde Matrix}

\begin{thm}
	Let $f(t)$ be a formal power series in $t$ with $f(0)=1$. Fix $\alpha_1,\ldots,\alpha_n\ge1$, let $m = \sum \alpha_i$, and $M(\alpha,m,x)$ be the following $\alpha \times m$ matrix:
	$$M(\alpha,m,x)_{i,j} := [t^{j-i}] \frac{f(t)}{(1-t x)^i}.$$
	Then we have:
	$$\det \begin{pmatrix}
	M(\alpha_1,m,x_1)\\
	\ldots\\
	M(\alpha_n,m,x_n)
	\end{pmatrix} = \prod_{i<j} (x_j-x_i)^{\alpha_i \alpha_j}.$$
\end{thm}
\begin{proof}
	This can be proven analogously to Theorem \ref{DoublyGeneralisedVandemonde}. The crucial insight is that
	\begin{align*}
	&M(\alpha,m,x)_{1,j}-M(\beta,m,y)_{1,j}
	= [t^{j-1}] f(t) \left(\frac{1}{1-t x}-\frac{1}{1-t y}\right)\\
	&= [t^{j-1}] f(t) \left(\frac{t(x-y)}{(1-t x) (1-t y)}\right)
	= (x-y) [t^{j-2}] \frac{ f(t)}{(1-t x) (1-t y)}
	\end{align*}
	and
	\begin{align*}
	&M(\alpha,m,x)_{i,j} - [t^{j-i}] \frac{ f(t)}{(1-t x)^{i-1} (1-t y)}\\
	&=[t^{j-i}] f(t) \left( \frac{1}{(1-t x)^i} - \frac{1}{(1-t x)^{i-1} (1-t y)} \right)\\
	&=(x-y) [t^{j-i-1}] \frac{ f(t)}{(1-t x)^i (1-t y)}.
	\end{align*}
\end{proof}

For $f(t)=1$ this reduces to the Generalised Vandermonde Matrix, as mentioned in \cite{10.2307/2690290}, which in turn reduces to the Vandermonde Matrix if all $\alpha_i=1$.

The case $f(t)=\prod_{i\in [k]} \frac{1}{1-t y_i}$ corresponds to Theorem \ref{DoublyGeneralisedVandemonde}.

For example, for $f(t)=\frac{1}{1-t}$ we get:
$$\det\left(
\begin{array}{ccccc}
1 & x_1+1 & x_1^2+x_1+1 & x_1^3+x_1^2+x_1+1 & x_1^4+x_1^3+x_1^2+x_1+1 \\
0 & 1 & 2 x_1+1 & 3 x_1^2+2 x_1+1 & 4 x_1^3+3 x_1^2+2 x_1+1 \\
1 & x_2+1 & x_2^2+x_2+1 & x_2^3+x_2^2+x_2+1 & x_2^4+x_2^3+x_2^2+x_2+1 \\
0 & 1 & 2 x_2+1 & 3 x_2^2+2 x_2+1 & 4 x_2^3+3 x_2^2+2 x_2+1 \\
0 & 0 & 1 & 3 x_2+1 & 6 x_2^2+3 x_2+1 \\
\end{array}
\right) = (x_2-x_1)^6.$$

\end{document}